\documentclass{amsart}
\usepackage{tikz}
\usepackage{amssymb}
\usepackage{stackrel}
\usepackage{color}
\usepackage{hyperref}
\usepackage{tikz-cd}
\usepackage{rotating}
\newcommand*{\isoarrow}[1]{\arrow[#1,"\rotatebox{90}{\(\simeq\)}"
	]}
\DeclareFontFamily{U}{wncy}{}
\DeclareFontShape{U}{wncy}{m}{n}{<->wncyr10}{}
\DeclareSymbolFont{mcy}{U}{wncy}{m}{n}
\DeclareMathSymbol{\Sha}{\mathord}{mcy}{"58}

\theoremstyle{plain}

\newtheorem*{theorem*}{Theorem}

\newtheorem{theorem}{Theorem}[section]
\newtheorem{lemma}[theorem]{Lemma}
\newtheorem{proposition}[theorem]{Proposition}
\newtheorem{conjecture}[theorem]{Conjecture}
\newtheorem{corollary}[theorem]{Corollary}
\theoremstyle{definition}
\newtheorem{definition}[theorem]{Definition}
\newtheorem{example}[theorem]{Example}

\newtheorem{remark}[theorem]{Remark}

\numberwithin{equation}{section}
\usepackage[all]{xy}
\DeclareMathOperator{\Po}{Po}

\DeclareMathOperator{\Cl}{Cl}
\DeclareMathOperator{\disc}{disc}
\DeclareMathOperator{\Gal}{Gal}
\DeclareMathOperator{\Int}{Int}
\DeclareMathOperator{\Ker}{Ker}

\DeclareMathOperator{\Ver}{Ver}

\DeclareMathOperator{\PSL}{PSL}

\begin{document}
	
	\title{On P\'olya groups of some non-Galois number fields}
	
	%\author{Amir Akbary}
	%\address{Department of Mathematics \& Computer Science
	%	University of Lethbridge,
	%	Lethbridge, Alberta, T1K 3M4}
	%\curraddr{}
	%\email{amir.akbary@uleth.ca}
	%\thanks{}
	
	\author{Abbas Maarefparvar}
	\address{Department of Mathematics \& Computer Science
		University of Lethbridge,
		Lethbridge, Alberta, T1K 3M4}
	\curraddr{}
	\email{abbas.maarefparvar@uleth.ca}
	\thanks{}
	\date{}
	\dedicatory{}
	\commby{}
	

	\begin{abstract}
		We prove two conjectures raised in  Chabert and Halberstadt's paper \cite{Chabert II}  concerning P\'olya groups of $S_4$-fields and $D_4$-fields. More generally, the latter will be proved for $D_n$-fields with $n \geq 4$ an even integer. Further, generalizing a result of Zantema \cite{Zantema}, we also prove that the pre-P\'olya group of a non-Galois field of a prime degree, e.g. an $A_5$-field, coincides with its P\'olya group.
%	We prove two conjectures raised in  Chabert and Halberstadt's paper \cite{Chabert II}  concerning P\'olya groups of $S_4$-fields and $D_4$-fields.
	%of non-Galois quartic number fields whose Galois closures have Galois group isomorphic to $D_4$ or $S_4$. Moreover, generalizing a result of Zantema \cite{Zantema}, we prove that  P\'olya group and pre-P\'olya group of a non-cyclic field of prime degree coincide.
	\end{abstract}

	\maketitle
	
	\vspace{.2cm} {\noindent \bf{Keywords:}}~ non-Galois number field, pre-P\'olya group, P\'olya group, $S_n$-field, $A_n$-field, $D_n$-field.

	\iffalse
	\[
	\begin{tikzcd}
		0 \arrow{r} & A \arrow{r} \isoarrow{d} & B \arrow{r} \isoarrow{d} & C \arrow{r} \isoarrow{d} & 0 \\ 
		0 \arrow{r} & A' \arrow{r} & B' \arrow{r} & C'  \arrow{r} & 0 
	\end{tikzcd}
	\]
	\fi
	
	\vspace{0.3cm}	
	
	{\noindent \bf{Notation.}}~  The following notations will be used throughout this article:

For a number field $K$, the notations $\Cl(K)$, $h_K$, $\mathcal{O}_K$, $\disc(K)$, and $\mathbb{P}_K$ denote the ideal class group, the class number, the ring of integers, the discriminant and the set of all prime ideals of $K$, respectively. For a finite extension $K/F$ of number fields, we use  $N_{K/F}$ to denote both the ideal norm and the element norm map from $K$ to $F$. The map ${\epsilon}_{K/F}:\left[\mathfrak{a}\right]\in \Cl(F) \rightarrow \left[\mathfrak{a} \mathcal{O}_K\right] \in \Cl(K)$ denotes the capitulation map. For an integer $n \geq 3$, we denote the symmetric group and the alternating group on $n$ symbols by $S_n$ and $A_n$, respectively. Also, $D_n$ denotes the dihedral group of order $2n$. %induced by the extension homomorphism from $I(F)$ to $I(K)$. %For a prime ideal $\mathfrak{p} \in \mathbb{P}_F$, we denote by $e_{\mathfrak{p}(K/F)}$   the ramification index of $\mathfrak{p}$ in $K/F$.
%For a number field $K$, the notations $I(K)$, $P(K)$, $\Cl(K)$, $h_K$, $\mathcal{O}_K$, $U_F$ and $\mathbb{P}_K$ denote the group of fractional ideals, group of principal fractional ideals, ideal class group, class number, ring of integers, group of units and set of all prime ideals of $K$, respectively. We also denote by $H_K$ the Hilbert class field of $K$. For a finite extension $K/F$ of number fields, we use  $N_{K/F}$ to denote both the ideal norm and the element norm map from $K$ to $F$. The map ${\epsilon}_{K/F}:\overline{\mathfrak{a}} \in \Cl(F) \rightarrow \overline{\mathfrak{a} \mathcal{O}_K} \in \Cl(K)$ denotes the capitulation map. %induced by the extension homomorphism from $I(F)$ to $I(K)$. %For a prime ideal $\mathfrak{p} \in \mathbb{P}_F$, we denote by $e_{\mathfrak{p}(K/F)}$   the ramification index of $\mathfrak{p}$ in $K/F$.

\vspace{.2cm}	

\section{Introduction} \label{section, background}
Let $K$ be a number field and denote its ring of integers by $\mathcal{O}_K$. The ring of integer-valued  polynomials on $\mathcal{O}_K$ is defined as
\begin{align*} 
	\Int(\mathcal{O}_K)=\{f \in K[x] \mid f(\mathcal{O}_K) \subseteq \mathcal{O}_K\}.
\end{align*}
It has been proved that $\Int(\mathcal{O}_K)$ is a free  $\mathcal{O}_K$-module, however finding an explicit  $\mathcal{O}_K$-basis for $\Int(\mathcal{O}_K)$ might be a challenging task \cite[Remark II.3.7]{Cahen-Chabert's book}. P\'olya \cite{Polya} was interested in those fields $K$ for which  Int($\mathcal{O}_K$) has a \textit{regular basis}, i.e.,  an $\mathcal{O}_K$ basis  $\{f_n\}_{n \geq 0}$ such that $deg(f_n)=n$, for every $n$. As suggested by Zantema \cite{Zantema}, such field $K$ is called a \textit{P\'olya field}.
 %Such basis, if it exists, is called a \textit{regular basis}. 

%For a non-negative integer $n$, denote by $\mathfrak{J}_n(K)$ the subset of $K$ formed by $0$ and the leading coefficients of polynomials in $\Int(\mathcal{O}_K)$ of degree $n$.  This is a fractional ideal of $\mathcal{O}_K$,  see
%\cite[Section 2]{Zantema}. P\'olya proved that $\Int(\mathcal{O}_K)$ has a regular basis if and only if all the ideals $\mathfrak{J}_n(K)$ are principal \cite[Satz I]{Polya}. In independent work, at the same time, in the same journal, and with the same title that P\'olya's paper was published, 
We observe that $K$ is a P\'olya field if and only if for every prime power $q$, the ideal 
\begin{align} \label{equation, Ostrowski ideal}
	\Pi_q(K)=:\prod_{\substack{\mathfrak{m}\in \mathbb{P}_K\\ N_{K/ \mathbb{Q}}(\mathfrak{m})=q}} \mathfrak{m}
\end{align}
is principal (By convention, if $K$ has no ideal with norm $q$,  set $\Pi_{q}(K)=\mathcal{O}_K$). This criterion is given by Ostrowski \cite{Ostrowski}, and so, we call $\Pi_q(K)$  an \textit{Ostrowski ideal}.

%\begin{definition} \cite{Zantema}
%	A number field $K$ is called a \textit{P\'olya field}, if any of the following equivalent conditions holds:
%	\begin{itemize}
	%	\item[(1)]
	%	Int($\mathcal{O}_K$) has a regular basis;
	%	\item[(2)]
	%	For every nonnegative integer $n$,  the ideal $\mathfrak{J}_n(K)$ is principal;
	%	\item[(3)]
	%	For every prime power $q$, the Ostrowski ideal $\Pi_q(K)$ is principal.
%	\end{itemize}
%\end{definition}

In a more modern approach, in 1997, the notion of \textit{P\'olya group} was introduced by Cahen and Chabert \cite[Chapter II, $\S$3]{Cahen-Chabert's book}.

\begin{definition} \label{definition, Polya group}
The P\'olya group of a number field $K$, denoted by $\Po(K)$, is the subgroup of $\Cl(K)$ generated by the classes of all the Ostrowski ideals \eqref{equation, Ostrowski ideal}, i.e., 
	\begin{equation} \label{equation, Polya group}
	\Po(K)=\langle \left[\Pi_{p^f}(K)\right]\, : \, f \in \mathbb{N}, \, \text{$p$ is a  prime number}\rangle.
\end{equation}
\end{definition}
%\begin{align} \label{equation, Ostrowski ideal}
%	\Pi_q(K)=:\prod_{\substack{\mathfrak{p}\in \mathbb{P}_K\\ N_{K/ %\mathbb{Q}}(\mathfrak{p})=q}} \mathfrak{p},
%%\end{align}
%where $q$ is a prime power \cite[Chapter II, Section 3]{Cahen-Chabert's book}. The number field $K$ is called a \textit{P\'olya field} whenever $\Po(K)=0$. %\begin{definition} \label{definition, Polya group}
%	The P\'olya group of a number field $K$ is the subgroup $\Po(K)$ of
%	the class group $\Cl(K)$ generated by the classes of all %the Ostrowski ideals $\Pi_q(K)$.
%\end{definition} 
%Hence $K$ is a P\'olya field if and only if $\Po(K)$ is trivial. 
%\begin{remark}
%Obviously every number field with class number one is a P\'olya field, but not conversely. For instance, Zantema proved that every cyclotomic field is a P\'olya field \cite[Proposition 2.6]{Zantema}.
%\end{remark}
 %Also, he obtained the following interesting result.
P\'olya fields, i.e., number fields with trivial P\'olya groups, include various ranges of number fields, such as class number one fields, cyclotomic fields \cite{Zantema}, and Hilbert class fields \cite{Leriche 2014}. %The classification of P\'olya fields, of a specific degree, has become of interest in recent years.

Investigating P\'olya fields and P\'olya groups can be divided into  Galois and non-Galois cases. If $K/\mathbb{Q}$ is Galois, then all the Ostrowski ideals above \textit{unramified} primes are principal \cite[Page 121]{Ostrowski}. Consequently, the P\'olya group $\Po(K)$ would be generated by the classes $\Pi_q(K)$ where $q$ is a power of a \textit{ramified} prime $p$ in $K/\mathbb{Q}$. Most of the results in the literature for P\'olya fields and P\'olya groups concern Galois number fields, see, e.g., \cite{Chabert I,Leriche 2011,Leriche 2013,Leriche 2014,MR1,MR2,M3}. In contrast, the situation for the non-Galois case is more complicated because, in this case, one also needs to consider the Ostrowski ideals above unramified primes. As the first tool for investigating the P\'olya-ness of non-Galois fields, the notion of \textit{pre-P\'olya field} was introduced by Zantema.

\begin{definition} \cite[Definition 6.1]{Zantema}
	A number field $K$ is called a pre-P\'olya field if for every unramified prime $p$ in $K/\mathbb{Q}$ and each $f \in \mathbb{N}$, the Ostrowski ideal $\Pi_{p^f}(K)$ \eqref{equation, Ostrowski ideal} is principal. 
\end{definition}

\begin{remark}
Note that if $K/\mathbb{Q}$ is Galois, then $K$ is a pre-P\'olya field. Because in this case, as mentioned before, all the Ostrowski ideals $\Pi_q(K)$ above unramified primes are principal.
\end{remark}

Let $[K:\mathbb{Q}]=n \geq 3$. Recall that for a transitive subgroup $G$ of $S_n$, $K$ is called a \textit{$G$-field} if the Galois closure of $K$ (over $\mathbb{Q}$) has the Galois group isomorphic to $G$. Zantema showed that for a `vast range' of non-Galois $G$-fields, pre-P\'olya-ness is equivalent to having class number one.

%For $A_n$-fields and $S_n$-fields, Zantema obtained the following result.

%\begin{definition} \cite{Zantema}
%	A number field $K$ of degree $n \geq 3$, is called a \textit{$G$-field} if its Galois  closure  (over $\mathbb{Q}$) has the Galois group isomorphic to $G$.
%For $n \geq 3$ an integer, let . %the symmetric group on $n$ symbols. 
%\end{definition}

\begin{theorem}[\textbf{Zantema}] \label{theorem, Zantema's result for Sn and An fields}  Let $K$ be a number field of degree $n \geq 3$. The following assertions hold:
	\begin{itemize}
		\item[(a)] 	\cite[Theorem 6.9]{Zantema} If $K$ is a $S_n$-field ($n \neq 4$), or an $A_n$-field ($n \neq 3,5$), or a $G$-field for some Frobenius group $G$, then
		\begin{equation*}
			\text{$K$ is a pre-P\'olya field}  \iff \text{$K$ is a P\'olya field} \iff h_K=1.
		\end{equation*}
		
		\item[(b)] \cite[Theorem 5.5]{Zantema2} If $K$ is a $S_4$-field, then
				\begin{equation*}
 \text{$K$ is a P\'olya field} \iff h_K=1.
		\end{equation*}
		
		\item[(c)] \cite[Section 8]{Zantema}  If $K$ is an $A_5$-field, then 
		\begin{equation*}
			\text{$K$ is a pre-P\'olya field}  \iff \text{$K$ is a P\'olya field}.
		\end{equation*}
	\end{itemize}
	
 %
	% $S$-field for one of the following groups $G$:
	%	\begin{itemize}
		%	\item $G=S_n$, for $n=3$ or $n \geq 5$;
		%	\item $G=A_n$, for $n=4$ or $n \geq 6$;
		%	\item $G$ is a Frobenius group.
		%	\end{itemize}
	%Then the following assertions are equivalent:
	%\begin{itemize}
	%	\item[(1)]
	%	$K$ is a pre-P\'olya field;
	%	\item[(2)]
	%	$K$ is a P\'olya field;
	%	\item[(3)] 
	%	$h_K=1$.
	%\end{itemize}
\end{theorem}

\begin{remark}
	Zantema gave an example of a pre-P\'olya $S_4$-fields which is not P\'olya, see \cite[Section 6]{Zantema}. Also, he presented an example of a P\'olya $A_5$-field with class number three, see \cite[Section 8]{Zantema}.
\end{remark}

Generalizing the notion of the pre-P\'olya field, Chabert and Halberstadt \cite{Chabert II} introduced the notion of \textit{pre-P\'olya group}.

%The notion of the pre-P\'olya field has been recently generalized to the \textit{pre-P\'olya group} by  Chabert and Halberstadt \cite{Chabert II}.

%\begin{definition} 
%	For a number field $K$, the \textit{pre-P\'olya group} of $K$, denoted by $\Po(K)_{nr}$, is the subgroup of $\Cl(K)$ generated by the classes of all the Ostrowski ideals above unramified primes, i.e.,
%	\begin{equation} \label{equation, Polya group}
	%	\Po(K)_{nr}=\langle \left[\Pi_{p^f}(K)\right]\, : \, f \in \mathbb{N}, \, \text{$p$ is an unramified prime in $K/\mathbb{Q}$}\rangle.
	%		\end{equation}
%\end{definition}

\begin{definition} \cite[Section 1]{Chabert II} \label{definition, pre-Polya group}
	For a number field $K$, the \textit{pre-P\'olya group} of $K$, denoted by $\Po(K)_{nr}$, is the subgroup of $\Cl(K)$ generated by the classes of all the Ostrowski ideals above unramified primes, i.e.,
	\begin{equation} \label{equation, pre-Polya group}
		\Po(K)_{nr}=\langle \left[\Pi_{p^f}(K)\right]\, : \, f \in \mathbb{N}, \, \text{$p$ is an unramified prime in $K/\mathbb{Q}$}\rangle.
	\end{equation}
\end{definition}

Part (a) of Theorem \ref{theorem, Zantema's result for Sn and An fields} has been generalized by Chabert and Halberstadt as follows. 

\begin{theorem} \cite[Theorem 4.1]{Chabert II}
	Let $K$ be a number field of degree $n \geq 3$. If $K$ is a $G$-field, where $G=S_n$ ($n \neq 4$), or  $G=A_n$ ($n \neq 3,5$), or $G$ is a Frobenius group, then
	\begin{equation*}
		\Po(K)_{nr1}=\Po(K)_{nr}=\Po(K)=\Cl(K),
	\end{equation*} 
where
\begin{equation} \label{equation, pre-Polya group1}
	\Po(K)_{nr1}=\langle \left[\Pi_{p}(K)\right]\, : \, \text{$p$ is an unramified prime in $K/\mathbb{Q}$}\rangle.
\end{equation}
\end{theorem}

	 %However, for these two exceptional cases, he proved the following results.

%For $S_4$-fields, one of the excluded case in Theorem \ref{theorem, Zantema's results for Sn and An fields}, Zantema proved the following result.

%For the two exceptional cases in Theorem \ref{theorem, Zantema's results for Sn and An fields}, namely for $S_4$-fields and $A_5$-fields, 
%Zantema obtained the following results.
%\begin{theorem}\cite[Theorem 5.5]{Zantema2} \label{theorem, Zantema's results for S4-fields}
%	Let $K$ be a $S_4$-field. Then $K$ is a P\'olya field if and only if $h_K=1$.
%\end{theorem}	

%\begin{theorem}\cite[Section 8]{Zantema} \label{theorem, Zantema's results for A5-fields}
%	Let $K$ be	an $A_5$-field. Then $K$ is a P\'olya field if and only if it is a pre-P\'olya field.
%\end{theorem}

Inspired by part (b) of Theorem \ref{theorem, Zantema's result for Sn and An fields}, 
Chabert and Halberstadt \cite[Section 6]{Chabert II} presented the following conjecture for $S_4$-fields.

\begin{conjecture}[\textbf{Chabert-Halberstadt}] \label{conj, S4-fields}

	Let $K$ be a $S_4$-field. Then 
	\begin{equation*}
		\Po(K)=\Cl(K).
	\end{equation*}
\end{conjecture}

Similarly, part (c) of Theorem \ref{theorem, Zantema's result for Sn and An fields} leads us to the following conjecture.

\begin{conjecture}  \label{conj, A5-fields}
	Let $K$ be an $A_5$-field. Then $\Po(K)_{nr}=\Po(K)$.
\end{conjecture}

Now let $K$ be a $D_n$-field (Recall that $D_n$ denotes the dihedral group of order $2n$). Chabert and Halberstadt proved that $\Po(K)_{nr1}=\Po(K)_{nr}$, see \cite[Proposition 5.1]{Chabert II}. Then, investigating $\Po(K)_{nr}$ in the special case $[K:\mathbb{Q}]=4$, they proposed the following conjecture.

\begin{conjecture}[\textbf{Chabert-Halberstadt}] \label{conj, D4-fields}
 \cite[Section 6]{Chabert II}
	Let $K$ be a $D_4$-field. If $h_K=2$, then $\Po(K)_{nr}=0$.
\end{conjecture}

The main goal of this paper is to prove the above three conjectures. For proving Conjectures \ref{conj, S4-fields} and \ref{conj, A5-fields}, we employ the Main Theorem (Theorem \ref{theorem, the composite maps}) in which we give a general tool that facilitates investigating P\'olya groups of non-Galois number fields.

\medspace

\noindent
\textbf{Main Theorem (Informal Version).} \textit{Let $K$ be a non-Galois number field of degree $n \geq 3$. Let $H_0=\Gal(N_K/K)$, where $N_K$ denotes the Galois closure of $K$ (over $\mathbb{Q}$). Then, there exists an explicit group homomorphism
	\begin{equation} \label{equation, hom Phi}
		\Phi : \Po(K) \rightarrow \frac{H_0}{\langle H_0' \rangle},
	\end{equation}
	with $\Ker(\Phi) \subseteq  \Po(K)_{nr1}$, where $H_0'$ denotes the derived subgroup of $H$.}

\medspace

The homomorphism in $\Phi$ \eqref{equation, hom Phi} is the key tool in proving Conjectures \ref{conj, S4-fields} and \ref{conj, A5-fields}. More precisely, in Section \ref{Section, S4 fields}, we prove the following theorem.

\medspace

\begin{theorem*}[\textbf{Theorem \ref{theorem, S4-field}}]
Let $K$ be a $S_4$-field. Then $\Po(K)=\Cl(K)$. In other words, Conjecture \ref{conj, S4-fields} is true.
\end{theorem*}
 %Using the Main Theorem, we show that Conjecture \ref{conj, S4-fields} is true, see Theorem \ref{theorem, S4-field}. 
As another application of Main Theorem, in Section \ref{Section, non-cyclic fields}, we obtain the following result in which  $\PSL(d,q)$ denotes the group of $d \times d$ matrices over $\mathbb{F}_q$ of determinant $1$ modulo scalars. Also, $\mathrm{P}\Sigma \mathrm{L}(d,q)$ denotes the permutation group on the points of $\mathbb{P}^{d-1}(\mathbb{F}_q)$ generated by $\text{Aut}(\mathbb{F}_q)$ and $\PSL(d,q)$, where $\text{Aut}(\mathbb{F}_q)$ denotes the group of field automorphisms of $\mathbb{F}_q$ \cite[Section 7]{Zantema}.

\medspace

\begin{theorem*}[\textbf{Theorem \ref{theorem, our results for non-cyclic fields of prime degrees}}]
Let $K/\mathbb{Q}$ be a non-Galois extension of prime degree $\ell \geq 5$. Denote by $N_K$ the Galois closure of $K$ (over $\mathbb{Q}$),  and let $G_0=\Gal(N_K/\mathbb{Q})$ be unsolvable. Then 
\begin{equation*}
	\Po(K)_{nr1}=\Po(K)_{nr}=\Po(K)=\Cl(K),
\end{equation*}
unless 
\begin{equation*}
	\PSL(2,q) \subseteq G_0 \subseteq \mathrm{P}\Sigma \mathrm{L}(2,q),
\end{equation*}
for some prime power $q$. In this case,
$\ell$  is a Fermat prime, $\frac{\Cl(K)}{\Po(K)_{nr1}}$ is a cyclic group whose order dividing $\ell -2$, and 
\begin{equation} \label{equation, Po(K)nr1 is equal to Po(K)-1}
	\Po(K)_{nr1}=\Po(K)_{nr}=\Po(K).
\end{equation} 
In particular, for a $G_0$-field $K$ with $G_0=A_5 \simeq \PSL(2,4)$, the equalities \eqref{equation, Po(K)nr1 is equal to Po(K)-1} hold. Consequently, Conjecture \ref{conj, A5-fields} is true.
\end{theorem*}

%prove that for \textit{all} non-cyclic fields $K$ of a prime degree $\ell$ whose Galois closure (over $\mathbb{Q}$) has an \textit{unsolvable} Galois group, we have $\Po(K)_{nr}=\Po(K)$, see Theorem \ref{theorem, our results for non-cyclic fields of prime degrees}. Consequently, we show that Conjecture \ref{conj, A5-fields} is true in a more general case. 
Finally, in Section \ref{Section, even dihedral fields}, using a different approach, namely using the decomposition form of primes in even dihedral number fields, we prove that Conjecture \ref{conj, D4-fields} holds more generally for all $D_n$-fields with $n \geq 4$ an even integer.

\medspace

\begin{theorem*}[\textbf{Theorem \ref{theorem, our results for Dn-fields}}]
	Let $K$ be a $D_n$-field, where $n \geq 4$ is an even integer. If $h_K=2$, then $\Po(K)_{nr}=0$. In particular, Conjecture  \ref{conj, D4-fields} is true.
\end{theorem*}

\vspace*{0.2cm}

\section{Image of P\'olya group under the Artin isomorphism} \label{section, image of Po(K)}

Zantema \cite[Section 6]{Zantema} gave a group theoretical interpretation of pre-P\'olyaness condition for number fields $K$. Using this approach, Chabert and Halberstadt \cite[Proposition 2.3]{Chabert II} determined the Artin symbols corresponding to the Ostrowski ideals $\Pi_{p^f}(K)$ above non-ramified primes $p$. In this section, we use Zantema's approach \cite[Section 8]{Zantema}, to describe the image of the P\'olya group $\Po(K)$ under the Artin isomorphism, in which we also consider the Ostrowski ideals above ramified primes. Then we present the main result of this paper, namely Theorem \ref{theorem, the composite maps}, which relates $\Po(K)$ to some quotient factor of the Galois group of the Galois closure of $K$ (over $\mathbb{Q}$).

%Let $K$ be a number field of degree $n \geq 3$. 
From now on, %as suggested in \cite[Section 3]{Chabert II}, 
we fix the following notations:
\begin{itemize}
	%\item[] $K$: a number field of degree $n \geq 3$ over $\mathbb{Q}$;
	\item[] $N_K$: the Galois closure of $K$ over $\mathbb{Q}$;
	\item[] $G_0=\Gal(N_K/\mathbb{Q})$;
	\item[] $H_0=\Gal(N_K/K)$;
	\item[] $H_K$: the Hilbert class field of $K$;
	\item[] $N_{H_K}$: the Galois closure of $H_K$ over $\mathbb{Q}$;
\item[] $G=\Gal(N_{H_K}/\mathbb{Q})$;
\item[] $H=\Gal(N_{H_K}/K)$;
\item[] $H_1=\Gal(N_{H_K}/H_K)$;
\item[] $H_2=\Gal(N_{H_K}/N_K)$.  
\end{itemize}

\begin{remark}
Note that $H_2=\cap_{g \in G} gH^{-1}g$, $G_0=G/H_2$, and $H_0=H/H_2$.
\end{remark}

Let $p$ be a prime number, and fix a prime ideal $\mathfrak{P}$ of $N_H$ above $p$. Let $Z(\mathfrak{P}/p)$ and $T(\mathfrak{P}/p)$ be the decomposition group and the inertia group of $N_H/\mathbb{Q}$ at $\mathfrak{P}$, respectively. The group $G$ acts on the set of $n$ cosets $Hs$ ($s \in G$) as
\begin{equation*}
	(Hs)^g:=Hsg^{-1}, \quad \forall g \in G,
\end{equation*}
see \cite[Section 8]{Zantema}. Under this action, let $g_1,g_2,\dots,g_t$ be the lengths of the orbits $Z(\mathfrak{P}/p)$. Since $T(\mathfrak{P}/p)$ is a normal subgroup of $Z(\mathfrak{P}/p)$, each orbit of $Z(\mathfrak{P}/p)$, say the $i$-th orbit of length $g_i$, splits into $f_i$ orbits of $T(\mathfrak{P}/p)$ of equal length, say $e_i$.

\begin{proposition} \cite[Section 8]{Zantema} \label{proposition, Zantema result for phii's and psii's}
For a prime number $p$, keep the notations as above. For every $i=1,2,\dots,t$, choose $s_i \in G$   such that $Hs_i$ is in the $i$-th orbit of $Z(\mathfrak{P}/p)$. Then the following assertions hold:
\begin{itemize}
	\item[(1)]
	 We have
	\begin{equation*}
		p \mathcal{O}_K=\prod_{i=1}^{t} (s_i(\mathfrak{P})\cap K)^{e_i}
	\end{equation*}
	is the decomposition form of $p$ in $K/\mathbb{Q}$, where $N_{K/\mathbb{Q}}(s_i(\mathfrak{P})\cap K)=p^{f_i}$.
	
	\vspace*{0.2cm}
	\item[(2)]  Let $\frac{Z(\mathfrak{P}/p}{T(\mathfrak{P}/p)}=\langle g_{\mathfrak{P}}\rangle $ for some $g_{\mathfrak{P}} \in G$. Then for every $i=1,2,\dots,t$, there exists $a_i \in T(\mathfrak{P}/p)$ such that
		\begin{equation*}
				s_ig_{\mathfrak{P}}^{f_i}a_is_i^{-1} \in H.
			\end{equation*}
	Moreover, for an  integer $f \geq 1$, the ideal $\Pi_{p^f}(K)$ is principal if and only if 
	\begin{equation*}
			\prod_{\{i | f_i=f\}}	s_ig_{\mathfrak{P}}^{f_i}a_is_i^{-1} \in H_1.
	\end{equation*}
%\item[(2)] For every $i=1,\dots,t$, consider the maps
%\begin{equation}
%\frac{	s_i Z(\mathfrak{P}/p)s_i^{-1}}{	s_i T(\mathfrak{P}/p)s_i^{-1}} \xrightarrow{\phi_i} \frac{	\left(s_i Z(\mathfrak{P}/p)s_i^{-1} H \right)}{\left(s_i T(\mathfrak{P}/p)s_i^{-1} H \right)} \xrightarrow{\psi_i} \frac{H}{H_1},
%\end{equation}
%	where $\phi_i$ is defined by 
%	\begin{equation*}
	%	\phi_i(x):=x^{f_i}, \quad \forall x \in \frac{	s_i Z(\mathfrak{P}/p)s_i^{-1}}{	s_i T(\mathfrak{P}/p)s_i^{-1}},
%	\end{equation*}
%	and $\psi_i$ is defined by the identity on $H$. Let $g_p$ be a generator of $\frac{Z(\mathfrak{P})/p}{T(\mathfrak{P})/p}$. Then under the Artin isomorphism $\Cl(K)\stackrel{\alpha_K}{\simeq} H/H_1$, one has
%	\begin{equation*}
	%	\alpha_K\left(\left[\Pi_{p^f}(K)\right]\right)= \prod_{\{i | f_i=f\}}	\psi_i \circ \phi_i \left(s_i g_p s_i^{-1}\right), \quad \forall f \in \mathbb{N}.
	%\end{equation*}
	
%	\vspace*{0.2cm}
%	\item[(3)]  Let $Z(\mathfrak{P}/p)/T(\mathfrak{P}/p)=<g>$ for some $g \in G$. Then for every $i=1,2,\dots,t$, there exists $a_i \in T(\mathfrak{P}/p)$ such that
%	\begin{equation*}
%		s_ig^{f_i}a_is_i^{-1} \in H.
%	\end{equation*}
%Moreover, for an  integer $f \geq 1$, the ideal $\Pi_{p^f}(K)$ is principal if and only if 
%\begin{equation*}
%		\prod_{\{i | f_i=f\}}	s_ig^{f_i}a_is_i^{-1} \in H_1.
%\end{equation*}
\end{itemize}

\end{proposition}

\begin{remark}
	If $p$ is unramified in $K/\mathbb{Q}$, then $T(\mathfrak{P}/p)$ is trivial. In this case, Proposition \ref{proposition, Zantema result for phii's and psii's} coincides with \cite[Proposition 2.3]{Chabert II}.
\end{remark}

Using Proposition \ref{proposition, Zantema result for phii's and psii's} and by an argument similar to \cite[Proof of Proposition 2.3]{Chabert II},  one can describe the Ostrowski ideals $\Pi_{p^f}(K)$ in terms of the Artin symbols.

\begin{corollary} \label{corollary, Po(K) under Artin map}
With the notations of this section, for every prime number $p$ and every integer $f \geq 1$, one has
\begin{equation*}
	\left(\frac{H_K/K}{\Pi_{p^f}(K)}\right)=\prod_{\{i | f_i=f\}}	s_ig_{\mathfrak{P}}^{f_i}a_is_i^{-1} |_{H_K}.
\end{equation*}
In other words, under the Artin isomorphism $\Cl(K)\stackrel{\alpha_K}{\simeq} H/H_1$, for every prime number $p$ and every integer $f 
\geq 1$, we have
\begin{equation}
	\alpha_K\left([\Pi_{p^f}(K)]\right)= \prod_{\{i | f_i=f\}}s_ig_{\mathfrak{P}}^{f_i}a_is_i^{-1} \, (\mathrm{mod}\, H_1) .
\end{equation}
%{\footnotesize 
%\begin{align*}
%	&\alpha_K \left(\Po(K)\right)= \\
%	&<\prod_{\{i | f_i=f\}}s_ig^{f_i}a_is_i^{-1} (\mathrm{mod}\, H_1) \, : \, \text{$\frac{Z(\mathfrak{P}/p)}{T(\mathfrak{P}/p)}=<g>$ for some $\mathfrak{P} \in \mathbb{P}_{N_H}$ above a prime number $p$}>.	
%\end{align*}} 

\end{corollary}

Zantema \cite[Section 6]{Zantema} showed that the Galois closures of ``some'' non-Galois pre-P\'olya fields coincide with the Galois closures of their Hilbert class fields.

\begin{theorem} \cite[Theorem 6.4]{Zantema} \label{theorem, Zantema's result when T is non-empty}
For a number field $K$ of degree $n \geq 3$, let $G$, $H$ and $H_1$ be defined as in this section. Under the action of $G$ on the set of $n$ cosets $Hs$ ($s \in G$), defined as $(Hs)^g:=Hsg^{-1}$, let
\begin{equation} \label{equation, set T}
T=\{h \in H \, : \, \forall s \in G \backslash H, \, Hsh \neq Hs \}.
\end{equation}
Assume that $T \neq \emptyset$ and $K$ is a pre-P\'olya field. Then 
 $T \subseteq H_1$, and the Hilbert class field of $K$ is contained in the Galois closure of $K$, i.e., 
  $H_K \subseteq N_K$.
  %, where $H_K$ and $N_K$ denote the Hilbert class field and the Galois closure of $K$ (over $\mathbb{Q}$), respectively. 

\end{theorem}

The above theorem, has been generalized by Chabert and Halberstadt (Recall that $\Po(K)_{nr1}$, as defined in \eqref{equation, pre-Polya group1}, is the subgroup of $\Cl(K)$ generated by all the classes of $\Pi_p(K)$ above unramified primes $p$).%, see Proposition \ref{proposition, Ch-H result for Po(K)nr1} below, for which they introduced the notion of $\Po(K)_{nr1}$.

%\begin{definition} \cite[Section 3]{Chabert II}
%	For a number field $K$, the notation $\Po(K)_{nr1}$ denotes the subgroup of $\Cl(K)$ generated by the classes of the Ostrowski ideals $\Pi_{p^1}(K)$ above unramified primes $p$:
%	\begin{equation*}
	%	\Po(K)_{nr1}=<\left[\Pi_{p^1}(K)\right] \, :\, \text{$p$ is an unramified prime in $K/\mathbb{Q}$}>.
%	\end{equation*} 
%\end{definition}

\begin{proposition} \label{proposition, Ch-H result for Po(K)nr1} \cite[Section 3]{Chabert II}
Let $K$ be a number field of degree $n \geq 3$. With the notations of this section, the following assertions hold:
\begin{itemize}
	\item[(1)]  Under the Artin isomorphism $\Cl(K)\stackrel{\alpha_K}{\simeq} H/H_1$, we have
	\begin{equation*}
		\alpha_K^{-1}\left(T \, (\mathrm{mod}\, H_1)\right)\subseteq \Po(K)_{nr1},
	\end{equation*}
where $T$ is the set defined in \eqref{equation, set T}.
	In particular, if $\Po(K)_{nr1}=0$, then $T \subseteq H_1$.
	\vspace*{0.2cm}
		\item[(2)]  If $T \neq \emptyset$, then $H_2 \subseteq \langle T\rangle $.
		\vspace*{0.2cm}
	\item[(3)]  The factor group $\frac{\Cl(K)}{\Po(K)_{nr1}}$  is isomorphic to a quotient group of $\frac{H}{\langle H_1,T\rangle }$.
\end{itemize}
\end{proposition}

%As we see, the above method is based on working with the Galois closure $N_H$ of the Hilbert class field of $K$ over $\mathbb{Q}$. 
As we can see, the set ``$T$'' is a key object in Theorem \ref{theorem, Zantema's result when T is non-empty} and Proposition \ref{proposition, Ch-H result for Po(K)nr1}, which is related to the action of $G$ on the cosets $Hs$. 
%results that are presented in this section so far, rely on the Galois groups related to the Galois closures of the Hilbert class field of $K$, say on $G=\Gal(N_{H_K}/\mathbb{Q})$ and $H=\Gal(N_{H_K}/K)$.
However, Chabert and Halberstadt \cite{Chabert II} showed that in some cases, one may consider the action of $G_0$ on the cosets $H_0s$, which makes the situation easier to investigate $\Po(K)$.
%Galois group of the Galois closure of $K$ itself, say $G_0=\Gal(N_K/\mathbb{Q})$ and $H_0=\Gal(N_K/K)$. Their observation makes it easier to investigate $\Po(K)$. %closure $N_H$ of $H_K$ with the Galois closure $N_K$ of $K$ (over $\mathbb{Q}$), i.e., replace  

\begin{theorem} \label{theorem, CH} \cite[Theorem 3.11 and Remark 3.12]{Chabert II}
	With the notations of this section, denote by $\Omega_0$ the set formed by the right cosets of $H_0$ in $G_0$. Define the action of $G_0$ on $\Omega_0$ as 
%Let $K$ be number field of degree $n \geq 3$ and $N_K$ be the Galois closure of $K$ over $\mathbb{Q}$. Let $G_0=\Gal(N_K/\mathbb{Q})$ and $H_0=\Gal(N_K/K)$. Denote by $\Omega_0$ the set formed by the right cosets of $H_0$ in $G_0$. Under the action of $G_0$ on $\Omega_0$, as 
\begin{equation*}
	(H_0s)^{g}:=H_0sg^{-1}, \quad \forall g \in G_0.
\end{equation*}
Let $T_0$ be the set formed by the elements of $H_0$, which have only $H_0$ as a fixed point. Then
\begin{equation*}
T_0 \neq \emptyset \iff T \neq \emptyset,
\end{equation*}
where $T$ is the set defined in \eqref{equation, set T}. Moreover, if $T_0 \neq \emptyset$, then
 $\frac{H}{\langle H_1,T\rangle }$ is a quotient group of  $\frac{H_0}{\langle T_0,H_0^{\prime}\rangle }$, where $H_0^{\prime}$ denotes the derived subgroup of $H_0$. Consequently, $\frac{\Cl(K)}{\Po(K)_{nr1}}$ is isomorphic to a quotient group of $\frac{H_0}{\langle T_0,H_0^{\prime}\rangle }$.
\end{theorem}

There are some applications of the above theorem for number fields of varying degrees, along with several examples for small degrees \cite[Sections 4--6]{Chabert II}. 
\begin{example} \cite[Section 6]{Chabert II}
For an $A_5$-field $K$ we have
\begin{equation*}
	G_0=A_5, \quad H_0=A_4, \quad H_0^{\prime}=\{1, (12)(34), (13)(24), (14)(23)\}, \quad T_0=H_0^{\prime} \backslash \{1\},
\end{equation*}
see \cite[proof of Proposition 6.8]{Chabert II}. By Theorem \ref{theorem, CH}, $|\frac{\Cl(K)}{\Po(K)_{nr1}}|$ divides $3$. 
\end{example}

\begin{remark}
	On one hand, by Theorem \ref{theorem, CH}, $\frac{H}{\langle H_1,T\rangle }$ is a quotient group of  $\frac{H_0}{\langle T_0,H_0^{\prime}\rangle }$. On the other hand,  $\frac{H_0}{\langle T_0,H_0^{\prime}\rangle }$ is an abelian group. Hence $\frac{H}{\langle H_1,T\rangle }$ would be isomorphic to a subgroup of $\frac{H_0}{\langle T_0,H_0^{\prime}\rangle }$. Consequently, there exists an injection 
	\begin{equation} \label{equation, injectition map lambda}
		\lambda:	\frac{H}{\langle H_1,T\rangle } \hookrightarrow \frac{H_0}{\langle T_0,H_0^{\prime}\rangle }.
	\end{equation}
	%hence every factor group of $\frac{H_0}{\langle T_0,H_0^{\prime}\rangle }$ can be considered as one of its subgroups. This explains why in Theorem \ref{theorem, the composite maps}, we assume that the map $\lambda$ is injective.
\end{remark}

%The above results %of Zantema \cite{Zantema}, and Chabert and Halberstadt \cite{Chabert II}, 
%have led us to obtain the main theorem of this paper as follows.
Inspired by the above results, we obtain the main theorem of this paper as follows.
\begin{theorem}[\textbf{Main Theorem}] \label{theorem, the composite maps}
	 Let $K$ be a number field of degree $n \geq 3$. %For a fixed prime number $p$, with the notations introduced in Section \ref{section, image of Po(K)}, for every $i=1,\dots,t$, 
With the notations of this section, under the action of $G$ (resp. $G_0$) on the right cosets of $H$ in $G$ (resp. $H_0$ in $G_0$), let $T$ (resp. $T_0$) be the set formed by the elements of $H$ (resp. $H_0$) which have only $H$ (resp. $H_0$) as fixed point. Consider the composite map
	\begin{equation} \label{equation, composite map Lambda o theta}
		\Po(K) \xrightarrow{\alpha_K} \frac{H}{H_1} \xrightarrow{\theta}  \frac{H}{\langle H_1,T\rangle } \xrightarrow{\lambda} \frac{H_0}{\langle T_0,H_0^{\prime}\rangle },
	\end{equation}
%\begin{equation} \label{equation, composite maps1}
%	\frac{	s_i Z(\mathfrak{P}/p)s_i^{-1}}{	s_i T(\mathfrak{P}/p)s_i^{-1}} \xrightarrow{\phi_i} \frac{	\left(s_i Z(\mathfrak{P}/p)s_i^{-1} H \right)}{\left(s_i T(\mathfrak{P}/p)s_i^{-1} H \right)} \xrightarrow{\psi_i} \frac{H}{H_1} \xrightarrow{\theta}  \frac{H}{<H_1,T>} \xrightarrow{\lambda} \frac{H_0}{<T_0,H_0^{\prime}>},
%\end{equation}
%where $\phi_i$ and $\psi_i$ are defined as in Proposition \ref{proposition, Zantema result for phii's and psii's}, the map $\theta$ is the canonical surjection, and  $\lambda$ is the embedding map obtained by Theorem \ref{theorem, CH}. Define the map
where  $\alpha_K$ denotes the Artin isomorphism, $\theta$ is the canonical surjection, and  $\lambda$ is the embedding map \eqref{equation, injectition map lambda}. %obtained from Theorem \ref{theorem, CH}. 
Define the map
\begin{equation} \label{equation, map Phi from Po(K) to H0/T0}
	\Phi : \Po(K)\rightarrow \frac{H_0}{\langle T_0,H_0^{\prime}\rangle }
\end{equation}
as 
\begin{equation*}
	\Phi:=\lambda \circ \theta \circ \alpha_K.
\end{equation*}
%\begin{equation}
%	\Phi\left(\left[\Pi_{p^f}(K)\right]\right):= \lambda \circ \theta \left(\prod_{\{i | f_i=f\}}	\psi_i \circ \phi_i \left(s_i g_p s_i^{-1}\right)\right) , \quad \forall f \in \mathbb{N}.
%\end{equation}
Then the following assertions hold:
\begin{itemize}
	\item[(i)]  For a prime number $p$, fix a prime ideal $\mathfrak{P}$ of $N_{H_K}$ above $p$. Let $Z(\mathfrak{P}/p)$ and $T(\mathfrak{P}/p)$ be the decomposition group and the inertia group of $N_{H_K}/\mathbb{Q}$ at $\mathfrak{P}$, respectively. Under the action of $G$ on the right cosets of $H$ in $G$, let $t$ be the number of the orbits of $Z(\mathfrak{P}/p)$, and let the $i$-th orbit splits into $f_i$ orbits of $T(\mathfrak{P}/p)$. For every $i\in\{1,2,\dots,t\}$, choose $s_i \in G$   such that $Hs_i$ is in the $i$-th orbit of $Z(\mathfrak{P}/p)$, and pick
	 $a_i \in T(\mathfrak{P}/p)$ such that
	\begin{equation*}
		s_ig_{\mathfrak{P}}^{f_i}a_is_i^{-1} \in H.
	\end{equation*} 
Then 
\begin{equation*}
	\left(s_i g_{\mathfrak{P}}^{f_i} a_i s_i^{-1} \right) |_{N_K}=\hat{s}_i g_{\mathfrak{P}_0}^{f_i} \hat{a}_i \hat{s}_i^{-1} \in H_0,\quad \text{for} \, \,  i=1,\dots,t,
\end{equation*}
where $\hat{s}_i=s_i |_{N_K}$, $\hat{a}_i=a_i |_{N_K}$, and $g_{\mathfrak{P}_0}$ is the  Frobenius element of $N_K/\mathbb{Q}$ at  $\mathfrak{P}_0:=\mathfrak{P} \cap N_K$.
%generator of $\frac{Z(\mathfrak{P}_0/p)}{T(\mathfrak{P}_0/p)}$ for the prime ideal $\mathfrak{P}_0:=\mathfrak{P} \cap N_K$ of $N_K$ ( $Z(\mathfrak{P}_0/p)$ and $T(\mathfrak{P}_0/p)$ denote the decomposition group and the inertia group of $N_K/\mathbb{Q}$ at $\mathfrak{P}_0$, respectively). 
Furthermore,
	\begin{equation*}
		\Phi\left([\Pi_{p^f}(K)]\right)=\prod_{\{i | f_i=f\}} \hat{s}_i g_{\mathfrak{P}_0}^{f_i} \hat{a}_i \hat{s}_i^{-1} \, (\mathrm{mod}\, \langle T_0, H_0^{\prime}\rangle ), \quad \forall f \in \mathbb{N}.
	\end{equation*}
	\vspace*{0.15cm}
	\item[(ii)] We have $\Ker(\Phi) \subseteq  \Po(K)_{nr1}$.

	%\begin{equation}
%\Phi\left(\left[\Pi_{p^f}(K)\right]\right)=\prod_{\{i | f_i=f\}}	%\tilde{s}_i \tilde{g}_p \tilde{s}_i^{-1},
%	\end{equation}
%	where $\tilde{s}_i=s_i|_{N_K}$ and $\tilde{g}_p$ is the generator of $\frac{Z(\mathfrak{P}_0/p)}{T(\mathfrak{P}_0/p)}$ for $\mathfrak{P}_0=\mathfrak{P} \cap N_K$.
\end{itemize}

\end{theorem}

\begin{proof}
	(1) Let $g_{\mathfrak{P}} \in G$ be the generator of $\frac{Z(\mathfrak{P}/p)}{T(\mathfrak{P}/p)}$. Using part (2) of Proposition \ref{proposition, Zantema result for phii's and psii's}, we have
	\begin{equation} \label{equation, image of Phi, first step}
	\Phi\left([\Pi_{p^f}(K)]\right)=\lambda \left(\prod_{\{i | f_i=f\}} s_i g_{\mathfrak{P}}^{f_i} a_i s_i^{-1} \, (\mathrm{mod}\, \langle H_1,T\rangle )\right).
	\end{equation}
	On the other hand, the injection map $\lambda$ is the composite map
	\begin{equation*}
		\frac{H}{\langle H_1,T\rangle } \hookrightarrow \frac{H}{\langle T,V\rangle } \simeq \frac{H/H_2}{\langle T,V\rangle /H_2} \simeq \frac{H_0}{\langle T_0,V_0\rangle } \hookrightarrow \frac{H_0}{\langle T_0,H_0^{\prime}\rangle },
	\end{equation*}
	where $V$ (resp. $V_0$) denotes the image of the \textit{transfer map} $\Ver_{G \rightarrow H}$ (resp. $\Ver_{G_0 \rightarrow H_0}$), see \cite[Proof of Theorem 3.11]{Chabert II}. Since $H_2=\Gal(N_H/N_K)$, considering the class of an element $s \in H$  modulo $H_2$ is equivalent to considering the restriction of $s$ to $N_K$. In particular,
	\begin{equation*}
	\left(s_i g_{\mathfrak{P}}^{f_i} a_i s_i^{-1} \right) |_{N_K} \in H/H_2=H_0, \quad \text{for} \, \, i=1,\dots,t,
	\end{equation*}
	and  \eqref{equation, image of Phi, first step} is equivalent to
	\begin{equation*}
		\prod_{\{i | f_i=f\}} \left(s_i g_{\mathfrak{P}}^{f_i} a_i s_i^{-1} \right) |_{N_K} \, (\mathrm{mod}\, \langle T_0,H_0^{\prime}\rangle ).
	\end{equation*}
 Finally, since $g_{\mathfrak{P}}$ is the generator of $\frac{Z(\mathfrak{P}/p)}{T(\mathfrak{P}/p)}$, we have
	\begin{equation*}
		g_{\mathfrak{P}}(x) \equiv x^p \, (\mathrm{mod}\, \mathfrak{P}), \quad \forall x \in \mathcal{O}_{N_H}.
	\end{equation*}
Thus for $\mathfrak{P}_0=\mathfrak{P} \cap N_K$, one has
	\begin{equation*}
	g_{\mathfrak{P}}(x) \equiv x^p \, (\mathrm{mod}\, \mathfrak{P}_0), \quad \forall x \in \mathcal{O}_{N_K},
\end{equation*}
which implies that $g_{\mathfrak{P}}|_{N_K}$ is the Frobenius element of $N_K/\mathbb{Q}$ at $\mathfrak{P}_0$.
	\vspace*{0.3cm}
	
(2)	Let $\left[\Pi_{p^f}(K)\right] \in \Ker(\Phi)$. %Then by Proposition \ref{proposition, Zantema result for phii's and psii's}, we have
	%\begin{equation*}
	%	\lambda \circ \theta \circ \alpha_K \left( \left[\Pi_{p^f}(K)\right] \right) \in <H_0^{\prime}, T_0>,
	%\end{equation*}
%where $\alpha_K :\Cl(K) \rightarrow \frac{H}{H_1}$ denotes the Artin isomorphism.
	Since $\lambda$ is injective, we obtain
	\begin{equation*}
	\left[\Pi_{p^f}(K)\right] \in \Ker\left(\theta \circ \alpha_K\right).
	\end{equation*}
%	\begin{equation*}
	%	\theta \left(\alpha_K \left([\Pi_{p^f}(K)] \right) \right) \in %<T,H_1>,
	%\end{equation*}
Thereby,
\begin{equation*}
\alpha_K \left([\Pi_{p^f}(K)] \right) \in \langle T \, (\mathrm{mod}\, H_1)\rangle .
\end{equation*}
By part (1) of Proposition \ref{proposition, Ch-H result for Po(K)nr1}, we get
\begin{equation*}
	[\Pi_{p^f}(K)] \in \alpha_K^{-1} \left(\langle T \, (\mathrm{mod}\, H_1)\rangle  \right) \subseteq \Po(K)_{nr1}.
\end{equation*}

\end{proof}

\vspace*{0.2cm}

\section{Application of the Main Theorem for $S_4$-fields} \label{Section, S4 fields}

As mentioned in Section \ref{section, background}, Chabert and Halberstadt \cite[Section 6]{Chabert II} conjectured that for a $S_4$-field $K$, we have $\Po(K)=\Cl(K)$, see Conjecture \ref{conj, S4-fields}. 
%\begin{equation*}
%	\text{``$\Po(K)=\Cl(K)$, \quad for a $S_4$-field $K$.''}
%\end{equation*}
In this section, we show that this conjecture is true. To accomplish this, in addition to Theorem \ref{theorem, the composite maps}, we need several lemmas.
%First, we give a general approach concerning the quotients $\Cl(K)/\Po(K)_{nr1}$.

\begin{lemma} \label{lemma, degree M1}
	Let $K$ be a number field of degree $n \geq 3$. % Using the the Galois correspondence and the Artin isomorphism $\Cl(K) \stackrel{\alpha_K}{\simeq} \Gal(H_K/K)$, 
Let $M_1/K$  be the sub-extension of $H_K/K$ corresponding to the subgroup $\Po(K)_{nr1}$, obtained via the Galois correspondence and the Artin isomorphism $\Cl(K) \stackrel{\alpha_K}{\simeq} \Gal(H_K/K)$.   %and $\Po(K)$ of $\Cl(K)$, respectively. 
	If the set $T_0$, as defined in Theorem \ref{theorem, CH}, is non-empty, then 
	\begin{equation*}
		[M_1 : K] \,  | \,  \# \frac{H_0}{\langle H_0^{\prime},T_0 \rangle}.
	\end{equation*}
\end{lemma}

\begin{proof}
	%Since $\Po(K)_{nr1}$ is corresponding to the Galois group $\Gal(H_K/M_1)$, we get the first assertion.  
	By Theorem \ref{theorem, CH}, the order of $\frac{\Cl(K)}{\Po(K)_{nr1}}$ is a divisor of the order of $\frac{H_0}{\langle H_0^{\prime},T_0\rangle }$. Now, the assertion follows from the isomorphisms below
	\begin{equation*}
\frac{\Cl(K)}{\Po(K)_{nr1}} \simeq	\frac{\Gal(H_K/K)}{\Gal(H_K/M_1)} \simeq \Gal(M_1/K).
	\end{equation*}
%is also dividing $\# \frac{H_0}{\langle H_0^{\prime},T_0 \rangle}$.
\end{proof}

\begin{remark} \label{remark, M1 over K is unramified}
	Note that by the Galois correspondence, $\Po(K)_{nr1}=0$ if and only if $M_1=H_K$. Also, since $M_1 \subseteq H_K$, $M_1/K$ is an unramified abelian extension.
\end{remark}

\begin{lemma} \label{lemma, Po(K)nr1 is equal to Ker(Phi) for S4-fields}
	Let $K$ be a $S_4$-field. Then $\Ker(\Phi)=\Po(K)_{nr1}$, where 
	\begin{equation*}
		\Phi : \Po(K) \rightarrow \frac{H_0}{\langle H_0^{\prime},T_0\rangle}
	\end{equation*}
	is the homomorphism map defined in Theorem \ref{theorem, the composite maps}.
\end{lemma}

\begin{proof}
	Let $p$ be a prime which is not ramified in $K/\mathbb{Q}$. By Theorem \ref{theorem, the composite maps} we have
	\begin{equation*}
		\Phi \left([\Pi_{p^1}(K)]\right)=\prod_{\{i | f_i=1\}} \hat{s}_i g_{\mathfrak{P}_0} \hat{s}_i^{-1} \in H_0,	
	\end{equation*}
	where $g_{\mathfrak{P}_0}$ is the Frobenius element of $N_K/K$ at some prime $\mathfrak{P}_0$ above $p$, and $\hat{s}_i=s_i\mid_{N_K} \in G_0$, for all $i$ (The elements $s_i$'s and the integers $f_i$'s are defined in Theorem \ref{theorem, the composite maps}). Following Chabert and Halberstadt's method \cite[Section 6]{Chabert II}, since $\Gal(N_K/K) =H_0 \simeq S_3$, it is enough to show that 
	\begin{equation*}
		\prod_{\{i | f_i=1\}} \hat{s}_i g_{\mathfrak{P}_0} \hat{s}_i^{-1} \in H_0^{\prime} \simeq A_3,
	\end{equation*}
	for $g_{\mathfrak{P}_0}=(1 \, 2)$ and $g_{\mathfrak{P}_0}=(1 \, 2 \, 3)$. As suggested in \cite[Proof of Proposition 6.1]{Chabert II}, for $g_{\mathfrak{P}_0}=(1 \, 2)$, we can take $f_1=f_2=1$, $s_1=1$, and $s_2=(3 \, 4)$. Then, with the notations of Theorem \ref{theorem, the composite maps}, we have
	\begin{equation*}
		\Phi \left([\Pi_{p^1}(K)]\right)=\lambda \circ \theta \circ \alpha_K (\Pi_{p^1}(K))= \lambda \circ \theta (\prod_{i=1}^2 s_i g_{\mathfrak{P}} s_i^{-1})=\lambda \circ \theta \left( \underbrace{\left( (1\,2) (3 \, 4) \right)^2}_{1} \right)=1.
	\end{equation*}
	For $g_{\mathfrak{P}_0}=(1 \, 2 \, 3)$, we can take $f_1=1$, and $s_1=1$. Then
	\begin{equation*}
		\Phi \left([\Pi_{p^1}(K)]\right)=\prod_{\{i | f_i=1\}} \hat{s}_i g_{\mathfrak{P}_0} \hat{s}_i^{-1}= \hat{s}_1 g_{\mathfrak{P}}\mid_{N_K} \hat{s}_1^{-1} =g_{\mathfrak{P}_0}=(1 \, 2\, 3) \in H_0^{\prime}.
	\end{equation*}
	Hence $\Po(K)_{nr1} \subseteq \Ker(\Phi)$. The opposite inclusion is part (ii) of Theorem \ref{theorem, the composite maps}.
\end{proof}

\begin{corollary} \label{corollary, Ker(Phi) is equal to T mod H1 for S4-fields}
	Let $K$ be a $S_4$-field. Then $\alpha_K\left(\Po(K)_{nr1}\right)=\langle T \, (\mathrm{mod}\, H_1) \rangle$.
\end{corollary}

\begin{proof}
	As we have seen in the proof of Theorem \ref{theorem, the composite maps}, we have
	\begin{equation*}
		\Ker(\Phi) \subseteq \alpha_K^{-1} \left(\langle T \, (\mathrm{mod}\, H_1) \rangle  \right) \subseteq \Po(K)_{nr1}.
	\end{equation*}
	Now, the assertion follows from Lemma \ref{lemma, Po(K)nr1 is equal to Ker(Phi) for S4-fields}.
\end{proof}

\begin{corollary} \label{corollary, Artin symbol for M1/K}
	Let $K$ be a $S_4$-field. Let $M_1/K$  be the subextension of $H_K/K$ corresponding to the subgroup $\Po(K)_{nr1}$. Then
	for every $f \in \mathbb{N}$, we have
	\begin{equation*}
		\theta \circ \alpha_K \left(\left[ \Pi_{p^f}(K)\right]\right)=\left(\frac{M_1/K}{\Pi_{p^f}(K)}\right)
	\end{equation*}
	where 
	\begin{equation*}
		\alpha_K:\Cl(K) \rightarrow \Gal(H_K/K)=\frac{H}{H_1}
	\end{equation*}
	denotes the Artin isomorphism, and
	\begin{equation*}
		\theta:\frac{H}{H_1} \rightarrow \frac{H}{\langle T,H_1\rangle }
	\end{equation*}
	is the canonical surjection. Also,
	$\left(\frac{M_1/K}{\Pi_{p^f}(K)}\right)$ denotes the Artin symbol of $M_1/K$ at $\Pi_{p^f}(K)$.
\end{corollary}

\begin{proof}
	On one hand, by Corollary \ref{corollary, Po(K) under Artin map}, we have
	{\small 
		\begin{equation*}
			\theta \left(\alpha_K\left([\Pi_{p^f}(K)]\right)\right)= \theta \left(\prod_{\{i | f_i=f\}}s_ig_{\mathfrak{P}}^{f_i}a_is_i^{-1} \, (\mathrm{mod}\, H_1) \right)=\prod_{\{i | f_i=f\}}s_ig_{\mathfrak{P}}^{f_i}a_is_i^{-1} \, \left(  \mathrm{mod} \, \langle T,H_1\rangle\right).
	\end{equation*}}
	On the other hand, since $M_1 \subseteq H_K$, by \cite[Proposition 1.6]{NChildress}, we get
	{\small 
		\begin{equation*}
			\left(\frac{M_1/K}{\Pi_{p^f}(K)}\right)=\left(\frac{H_K/K}{\Pi_{p^f}(K)}\right)|_{M_1}=\prod_{\{i | f_i=f\}}	s_ig_{\mathfrak{P}}^{f_i}a_is_i^{-1} |_{M_1}=\prod_{\{i | f_i=f\}}	s_ig_{\mathfrak{P}}^{f_i}a_is_i^{-1} \,\left(  \mathrm{mod} \, \langle T,H_1\rangle\right).
	\end{equation*}}
Note that using Corollary \ref{corollary, Ker(Phi) is equal to T mod H1 for S4-fields}, the latter equality follows from
	\begin{equation*}
		\Gal(M_1/K)=\frac{\Gal(H_K/K)}{\Gal(H_K/M_1)}\simeq \frac{\alpha_K(\Cl(K))}{\alpha_K(\Po(K)_{nr1})}=\frac{H/H_1}{\langle T \,  (\mathrm{mod}\, H_1) \rangle}=\frac{H}{\langle T,H_1 \rangle}.
	\end{equation*}
\end{proof}

%For $K$, a $S_4$-field, Chabert and Halberstadt proved that $\Po(K)_{nr1}=0$ if and only if either $H_K=K$ or $H_K=K(\sqrt{\disc})$ \cite[Proposition 6.1]{Chabert II}. More generally, we prove:

\begin{lemma} \label{lemma. M1 is either K or Kdisc for S4-fields}
	Let $K$ be a $S_4$-field.  %Using the the Galois correspondence and the Artin isomorphism $\Cl(K) \stackrel{\alpha_K}{\simeq} \Gal(H_K/K)$, 
	Let $M_1/K$  be the sub-extension of $H_K/K$ corresponding to the subgroup $\Po(K)_{nr1}$, obtained via the Galois correspondence and the Artin isomorphism $\Cl(K) \stackrel{\alpha_K}{\simeq} \Gal(H_K/K)$. Then either $M_1=K$ or $M_1=K(\sqrt{\disc(K)})$.
\end{lemma}

\begin{proof}
	It has been proved in \cite[Proposition 6.1]{Chabert II} that $\# \frac{H_0}{\langle H_0^{\prime},T_0 \rangle} =2$. Hence by Lemma \ref{lemma, degree M1}, $[M_1:K]$ divides $2$. Let $[M_1:K]=2$. %We show that $M_1=K(\sqrt{\disc(K)})$. 
	Since $G_0=\Gal(N_K/\mathbb{Q})\simeq S_4$, we have
	\begin{itemize}
		\item $\sqrt{\disc(K)} \in N_K \backslash K$;
		
		\item $H_0 =\Gal(N_K/K) \simeq S_3$;
		
		\item $T_0 =\{( 1\, 2\, 3), (1 \, 3 \, 2)\} \subseteq H_0^{\prime} \simeq A_3$,
	\end{itemize}
	see \cite[Proof of Proposition 6.1]{Chabert II}. Let $F=K(\sqrt{\disc(K)})$.  Then $F/K$ is the unique quadratic subextension of $N_K/K$, which is the fixed subfield of $N_K$ by $H_0^{\prime}$.  %By Theorem \ref{theorem, the composite maps} we have
	Since, by Lemma \ref{lemma, Po(K)nr1 is equal to Ker(Phi) for S4-fields}, $\Ker(\Phi)=\Po(K)_{nr1}$,  we have
	\begin{equation*} \label{equation, Phi(Ponr1)}
		\Phi(\Po(K)_{nr1})%=\lambda \circ \theta \circ \alpha_K (\Po(K)_{nr1}) 
		\subseteq  \langle H_0^{\prime},T_0 \rangle =H_0^{\prime}.
	\end{equation*}
	Hence $M_1$ (corresponding to $\Po(K)_{nr1}$) would be isomorphic to a subfield of $F$. %(the fixed field by $H_0^{\prime}$). 
	But $[M_1:K]=[F:K]=2$, which implies that
	\begin{equation*}
		M_1 \simeq F=K(\sqrt{\disc(K)}).
	\end{equation*}
\end{proof}

\begin{remark}
Lemma \ref{lemma. M1 is either K or Kdisc for S4-fields} is a generalization of Chabert and Halberstadt's result \cite[Proposition 6.1]{Chabert II} in which they proved that $\Po(K)_{nr}=0$ if and only if either we have $H_K=K$ or $H_K=K(\sqrt{\disc(K)})$. Because, as mentioned in Remark \ref{remark, M1 over K is unramified}, $\Po(K)_{nr}=0$ if and only if $M_1=H_K$.
\end{remark}

\begin{lemma} \label{lemma, NK/K is unramified for S4 fields K}
	Let $K$ be a $S_4$ field, and $N_K$ be its Galois closure over $\mathbb{Q}$. Let $M_1/K$  be the subextension of $H_K/K$ corresponding to the subgroup $\Po(K)_{nr1}$. If $M_1=K(\sqrt{\disc(K)})$, then $N_K/K$ is unramified at all  finite primes.
\end{lemma}

\begin{proof}
	From the proof of Lemma \ref{lemma. M1 is either K or Kdisc for S4-fields}, we have
	\begin{equation*}
		\Gal(N_K)/K \simeq S_3,
	\end{equation*}
	and $M_1/K=K(\sqrt{\disc(K)})/K$ is the unique quadratic subextension of $N_K/K$. Let $L/K$ be a cubic subextension of $N_K/K$. Then, $L/K$ is non-Galois, so $\disc(L/K)$ is not square of any ideal of $K$. More precisely, the discriminant of $L/K$ can be written as
	\begin{equation} \label{equation, disc(L/K) for a cubic non-Galois L/K}
		\disc(L/K)=\disc(M_1/K).f^2,
	\end{equation}
	where $f$ is the conductor of $N_K$ over $M_1$, %Moreover, a prime $\mathfrak{p}$ of $K$ totally ramifies in $L$ if and only if $\mathfrak{p} \mid f$, 
	see \cite[Proposition 10.1.28]{H. Cohen Adv.}. Since $M_1/K$ is unramified, see Remark \ref{remark, M1 over K is unramified}, $\disc(M_1/K)$ is a unit in $K$. Because $\disc(L/K)$ can not be a perfect square, by \eqref{equation, disc(L/K) for a cubic non-Galois L/K}, the only possible case is that $f$ be also unit. Hence, $L/K$ is unramified at all finite primes of $K$. Since $N_K$ is the Galois closure of $L$ over $K$, $N_K/K$ is also unramified at all finite primes, as claimed.
\end{proof}

\begin{lemma} \cite[Lemma 5.4]{Zantema2} \label{lemma, Zantema's result for existence of E abelian}
	Let F be a number field and $N_F$ be its Galois closure over $\mathbb{Q}$. Let $P$ be a set of prime numbers $p$ for which the decomposition group in $\Gal(N_F/\mathbb{Q})$ (defined up to conjugacy) is abelian, and for which $N_F/F$ is unramified at all primes above $p$. Then there exists an abelian number field $E$ of degree
	\begin{equation*}
		[E:\mathbb{Q}]=\prod_{p \in P} e_{p(N_F/\mathbb{Q})},
	\end{equation*}
	where $ e_{p(N_F/\mathbb{Q})}$ denotes the ramification index of $p$ in $N_F/\mathbb{Q}$. Moreover, $(K.E)/K$ is unramified at all finite primes.
\end{lemma}

We are ready to prove Conjecture \ref{conj, S4-fields}.

\begin{theorem} \label{theorem, S4-field}
	Let $K$ be a $S_4$-field. Then $\Po(K)=\Cl(K)$.
\end{theorem}

\begin{proof}
	The proof is a generalization of Zantema's method in \cite[Proof of Theorem 5.5]{Zantema2}. Using the  Galois correspondence and the Artin isomorphism 
	\begin{equation*}
		\alpha_K:\Cl(K) \rightarrow \Gal(H_K/K)=\frac{H}{H_1},
	\end{equation*}
	let $M_1/K$ and $M_2/K$ be the subextensions of $H_K/K$ corresponding to the subgroups $\Po(K)_{nr1}$ and $\Po(K)$, respectively. We claim that $M_2=K$ (hence $\Po(K)=\Cl(K)$). If $M_1=K$ then $M_2=K$, which implies that
	\begin{equation*}
		\Po(K)_{nr1}=\Po(K)=\Cl(K).
	\end{equation*}
	So, using Lemma \ref{lemma. M1 is either K or Kdisc for S4-fields},  we only need to consider the case $M_1=K(\sqrt{\disc(K)})$. %In this case, % either $M_2=K$ or $M_2=M_1$.
	%we claim that $M_2=K$ (hence $\Po(K)=\Cl(K)$):
	To reach a contradiction, assume that $M_2=M_1$. Hence $\Po(K)=\Po(K)_{nr1}$. By Lemma \ref{lemma, Po(K)nr1 is equal to Ker(Phi) for S4-fields}, the map 
	\begin{equation*}
		\Phi : \Po(K) \rightarrow \frac{H_0}{\langle H_0^{\prime},T_0\rangle},
	\end{equation*}
	as defined in Theorem \ref{theorem, the composite maps}, is zero.
	Let $p$ be a prime number, and $\mathfrak{P}$ be a prime of $N_K$ above $p$. 
	%Following Zantema \cite[Proof of Theorem 5.5]{Zantema2}, denote the decomposition form $p \mathcal{O}_K=\prod_{i=1}^{g} \mathfrak{p}_i^{e_i}$ as
	%\begin{equation} \label{equation, decomposition from in Zantema's paper}
	%	f_1^{e_1} f_2^{e_2} \dots f_g^{e_g},
	%\end{equation}
	%where $f_i=N_{K/\mathbb{Q}}(\mathfrak{p}_i)$, for all $i=1,\dots,g$.
	In \cite[Section 3]{Zantema2}, Zantema determined all the possible cases, namely 18 cases, for the decomposition group $Z(\mathfrak{P}/p)$ and the inertia group $T(\mathfrak{P}/p)$ (as subgroups of $G_0=\Gal(N_K/\mathbb{Q})\simeq S_4$ up to conjugacy), see Table \ref{tab, Zantema's table} below.

	\begin{table}[!h]
		\begin{center}
			\begin{tabular}{|c | c|c|c | p{2.4cm} |p{2.7cm}| p{1.9cm}|} 
				\hline
				Case & $T(\mathfrak{P}/p)$ & $Z(\mathfrak{P}/p)$ & $\prod_{i=1}^n f_i^{e_i}$ & Condition on $p$ & Is $p$ unramified in $K(\sqrt{\disc(K)}/K)$? & Is $\left(\frac{M_1/K}{\Pi_{p^f}(K)}\right)$ trivial for all $f \in \mathbb{N}$? \\ [0.5ex] 
				\hline	
				1 & $\langle  (1 3)(2 4)\rangle$ & $\langle (1234) \rangle$ &	$2^2$ & -- & yes & yes \\
				\hline	
				2 & $\langle  (1 3)(2 4)\rangle$ & $\langle (1 3),(2 4) \rangle$ &	$1^2 1^2$ & -- & yes & yes \\	
					\hline	
				3 & $\langle  (1 3)(2 4)\rangle$ & $V_4$ &	$2^2$ & -- & yes & yes \\
				\hline	
				4 & $\langle  (1 3)(2 4)\rangle$ & $\langle  (1 3)(2 4)\rangle$ &	$1^2 1^2$ & -- & yes & yes \\
				\hline	
					5 & $\langle  (1 3)\rangle$ & $\langle  (1 3),(2 4)\rangle$ &	$1^2 2$ & -- & no & -- \\
				\hline	
					6 & $\langle  (1 3)\rangle$ & $\langle  (1 3) \rangle$ &	$1^2 11$ & -- & no & -- \\
				\hline
					7 & $\langle  (1 2 3)\rangle$ & $S_3$ &	$1^3  1$ & $p \equiv 0,2 \, (\mathrm{mod}\, 3)$ & yes & yes \\
				\hline
					8 & $\langle  (1 2 3)\rangle$ & $\langle  (1 2 3)\rangle$ &	$1^3  1$ & $p \equiv 0,1 \, (\mathrm{mod}\, 3)$ & yes & yes \\
				\hline
					9 & $\langle  (1 2 3 4)\rangle$ & $D_4$ &	$1^4$ & $p \equiv 2,3 \, (\mathrm{mod}\, 4)$ & yes & no \\
				\hline
					10 & $\langle  (1 2 3 4)\rangle$ & $\langle  (1 2 3 4)\rangle$ &	$1^4$ & $p \equiv 1,2 \, (\mathrm{mod}\, 4)$ & yes & yes \\
				\hline
					11 & $V_4$ & $A_4$ &	$1^4$ & $p=2$ & yes & yes \\
				\hline
					12 & $V_4$ & $D_4$ &	$1^4$ & $p=2$ & yes & no \\
				\hline
					13 & $V_4$ & $V_4$ &	$1^4$ & $p=2$ & yes & yes \\
				\hline
					14 & $\langle  (1 2), (3 4)\rangle$ & $D_4$ &	$2^2$ & $p=2$ & no & -- \\
				\hline
					15 & $\langle  (1 2), (3 4)\rangle$ & $\langle  (1 2), (3 4)\rangle$ &	$1^2 1^2$ & $p=2$ & no & -- \\
				\hline
					16 & $S_3$ & $S_3$ &	$1^3 1$ & $p=3$ & no & -- \\
				\hline
					17 & $D_4$ & $D_4$ &	$1^4$ & $p=2$ & no & -- \\
				\hline
					18 & $A_4$ & $S_4$ &	$1^4$ & $p=2$ & yes & no \\
				\hline
			\end{tabular} 	\caption{A detailed analysis of ramification in $S_4$-fields $K$, presented by Zantema  \cite[the table on page 100]{Zantema2}. The notation $V_4$ denotes the Klein four-group. Also, the product $\prod_{i=1}^n f_i^{e_i}$ is corresponded to the decomposition form $p \mathcal{O}_K=\prod_{i=1}^n  \mathfrak{P}_i^{e(\mathfrak{P}_i/p)}$, where $f_i=f(\mathfrak{P}_i/p)$ and $e_i=e(\mathfrak{P}_i/p)$.} \label{tab, Zantema's table} 
		\end{center}
	\end{table}
	%\begin{figure}[!p]
	%	\includegraphics[width=14cm,height=21cm]{6.jpg}
	%	\caption{A detailed analysis of ramification in $S_4$-fields $K$, presented by Zantema  \cite[the table on page 100]{Zantema2}} \label{fig, Zantema's table}
	%\end{figure}
Corresponding to each case, Zantema also determined if $\mathfrak{P}_0:=\mathfrak{P} \cap K$ is ramified in $M_1/K$, and checked the triviality of the Artin symbols $\left(\frac{M_1/K}{\Pi_{p^f}(K)}\right)$, for all $f \in \mathbb{N}$. Since $M_1/K$ is unramified, see Remark \ref{remark, M1 over K is unramified}, cases $5$, $6$, $14$, $15$, $16$, and $17$, among all the possible cases presented in Table \ref{tab, Zantema's table}, cannot occur. In particular, $K$ is either totally real or totally complex \cite[Proof of Theorem 5.5]{Zantema2}.  %It worth noting that the number field $K$ is either totally real or totally complex \cite[Proof of Theorem 5.5]{Zantema2}. 
	In cases $9,12,18$ the Artin symbol $\left(\frac{M_1/K}{\Pi_{p^f}(K)}\right)$ is non-trivial. By Corollary \ref{corollary, Artin symbol for M1/K}, for these cases, we have
	\begin{equation*}
		\Phi([\Pi_{p}(K)])=\lambda \left(\theta \circ \alpha_K \left(\left[ \Pi_{p^f}(K)\right]\right)\right) \neq 0 \, \, \text{in}\, \, \frac{H_0}{\langle H_0^{\prime},T_0 \rangle},
	\end{equation*}
	which is a contradiction since $\Phi$ is the zero map (Note that the map $\lambda$ is injective, see Theorem \ref{theorem, the composite maps}). 
	Let $F=\mathbb{Q}(\sqrt{\disc(K)})$ and $p_0$ be a prime dividing $\disc(F)$ (Note that $\Gal(N_K/F) \simeq A_4$). Then  $A_4$ does not contain $T(\mathfrak{P}/p_0)$ (This happens only for the case $10$ in Table \ref{tab, Zantema's table}). In particular, $e_{p_{0}(K/\mathbb{Q})}=4$. Since $M_1=K(\sqrt{\disc(K)})$, Lemma \ref{lemma, NK/K is unramified for S4 fields K} shows that $N_K/K$ is unramified at all finite primes. Hence, among the remaining cases, for those with abelian decomposition groups $Z(\mathfrak{P}/p)$, namely cases $1$, $2$, $3$, $4$, $10$, and $13$, we can apply Lemma \ref{lemma, Zantema's result for existence of E abelian} to get an abelian number field $E$ such that 
	\begin{equation*}
		[E:\mathbb{Q}]=\prod_{p \in P} e_{p(N_K/\mathbb{Q})},
	\end{equation*}
	where $P$ denotes the set of all primes $p$ whose decomposition group in $N_K/\mathbb{Q}$ is abelian. Following Zantema, we define the number field $M$ as
	\begin{equation} \label{equation, define M as E.K or ER.K}
		M= \left\{
		\begin{array}{ll}
			E\cdot K, &  \text{if $K$ is totally complex,}\\
			(E \cap \mathbb{R})\cdot K,  &  \text{if $K$ is totally real.} \\
		\end{array}
		\right.
	\end{equation}
	By Lemma \ref{lemma, Zantema's result for existence of E abelian}, $(M\cdot M_2)/K$ is an abelian extension which is unramified at all the finite places of $K$ (Recall that $M_2/K$ is the subextension of $H_K/K$ corresponded to $\Po(K)$). Hence
	\begin{equation*}
		M_2 \subseteq (M\cdot M_2) \subseteq H_K,
	\end{equation*}
	which implies that 
	\begin{equation*}
		\Gal(H_K/(M\cdot M_2)) \subseteq \Gal(H_K/M_2) \simeq \Po(K).
	\end{equation*}
	Let $\mathcal{A}$ be the subgroup of $\Po(K)$ corresponded to $\Gal(H_K/(M\cdot M_2))$. Since 
	\begin{equation*}
		\Phi : \Po(K) \rightarrow \frac{H_0}{\langle H_0^{\prime},T_0\rangle},
	\end{equation*}
	as defined in Theorem \ref{theorem, the composite maps}, is the zero map, we get
	\begin{equation} \label{equation, Phi of A is containded in H0}
		\Phi\left(\mathcal{A}\right) \subseteq \langle H_0^{\prime},T_0 \rangle =H_0^{\prime}.
	\end{equation}
	On one hand, by our assumption, $M_2=M_1=K(\sqrt{\disc(K)})$. On the other hand, since $K(\sqrt{\disc(K)})/K$ is the unique quadratic subextension of $N_K/K$ (that is the fixed subfield of $N_K$ by $H_0^{\prime}$), the relation \eqref{equation, Phi of A is containded in H0} yields $(M\cdot M_2)$, corresponding to the subgroup $\mathcal{A}$, would be  a subfield of $K(\sqrt{\disc(K)})=M_2$. In other words,
	\begin{equation*}
		M \subseteq M_2=K(\sqrt{\disc(K)}).
	\end{equation*}
	%Suppose that $E.K \not\subseteq M_2$. On one hand, by Lemma \ref{lemma, Zantema's result for existence of E abelian}, $E.K/K$ is an abelian unramified extension. Hence $E.K \subseteq H_K$. Since also $M_2 \subseteq H_K$, we have
	%\begin{equation} \label{equation, dim E.K is equal to order of Po(K)}
	%	\prod_{p \in P} e_{p(N_K/\mathbb{Q})}=[E.K:K]=[H_K:M_2]=\# \Po(K).
	%\end{equation}
	%we get $\prod_{p \in P} e_{p(N_K/\mathbb{Q})}=3$, which is a contradiction. Because, all the ramification indices $e_{p(N_K/\mathbb{Q})}$, as in cases $1$, $2$, $3$, $4$, $10$, and $13$, are even. So, we must have $E.K \subseteq M_2$, as claimed. Therefore,
	%Then one can show that $M \subseteq M_2=K(\sqrt{\disc(K)})$  
	Hence,
	\begin{align*}
		2=[M_2:K] \geq [M:K] \geq \frac{1}{2} [(E\cdot K):K]=\frac{1}{2}[E:\mathbb{Q}]=\frac{1}{2} \prod_{p \in P} e_{p(N_K/\mathbb{Q})} \\
		\geq \frac{1}{2} (e_{p_{0}(K/\mathbb{Q})}) \prod_{p_0 \neq p \in P} e_{p(N_K/\mathbb{Q})} = 2 \prod_{p_0 \neq p \in P} e_{p(N_K/\mathbb{Q})}.
	\end{align*}
	Therefore,
	\begin{equation*}
		\prod_{p_0 \neq p \in P} e_{p(N_K/\mathbb{Q})}=1,
	\end{equation*}
	which implies that cases $1$, $2$, $3$, $4$, and $13$ can not occur. Moreover, 
	\begin{equation*}
		[M:K]=\frac{1}{2} [(E \cdot K):K].
	\end{equation*}
	From \eqref{equation, define M as E.K or ER.K}, we conclude that $M=(E \cap \mathbb{R})\cdot K$. Thus, $E$ is complex, while $K$ is a totally real number field. Now, as it is shown by Zantema \cite[Pages 104--105]{Zantema2}, we see that none of the remaining cases, namely cases $7$, $8$, $10$, and $11$ in Table \ref{tab, Zantema's table}, can happen. Consequently, in the case $M_1=K(\sqrt{\disc(K)})$ we must have $M_2=K$, i.e., $\Po(K)=\Cl(K)$.
\end{proof}

\vspace*{0.2cm}
\section{Application of the Main Theorem for $A_5$-fields} \label{Section, non-cyclic fields}
%%\subsection{On Class groups of non-cyclic fields of prime degrees}
Let $K/\mathbb{Q}$ be a non-Galois extension of a prime degree, for which the Galois group of its Galois closure is an \textit{unsolvable} transitive permutation group. As an application of Theorem \ref{theorem, the composite maps}, we show that for such fields $K$, including $A_5$-fields, the subgroup $\Po(K)_{nr1}$ coincides with $\Po(K)$, %which can be even equal to $\Cl(K)$, 
see Theorem \ref{theorem, our results for non-cyclic fields of prime degrees} below. We begin by summarizing some results of Zantema (As mentioned in Section \ref{section, background}, the notions of $\PSL(d,q)$ and $\mathrm{P}\Sigma \mathrm{L}(d,q)$ denote the group of $d \times d$ matrices over $\mathbb{F}_q$ of determinant $1$ modulo scalars, and the permutation group on the points of $\mathbb{P}^{d-1}(\mathbb{F}_q)$ generated by $\text{Aut}(\mathbb{F}_q)$ and $\PSL(d,q)$, respectively).
%Keep the same notations as introduced in Section \ref{section, image of Po(K)}. Using some results of Zantema, one may find the quotient factor $\frac{H_0}{<H_0^{\prime},T_0\rangle }$ in the case that $K/\mathbb{Q}$ is a non-cyclic extension of a prime degree. Then, applying Theorem \ref{theorem, CH}, one determines $\frac{\Cl(K)}{\Po(K)_{nr1}}$ for such class of number fields $K$.

\begin{proposition} \cite[Sections 7]{Zantema} \label{proposition, possible groups G for non-cyclic}
Let $K/\mathbb{Q}$ be a non-Galois extension of prime degree $\ell \geq 5$. Denote by $N_K$ the Galois closure of $K$ (over $\mathbb{Q}$),  and assume that $G_0=\Gal(N_K/\mathbb{Q})$ is unsolvable.
Then, all the possible cases for $G_0$ can be listed as follows:
\begin{itemize}
	\item[(i)] $G_0=A_{\ell}$, or $G_0=S_{\ell}$;
	
	\item[(ii)] $G_0=\PSL(2,11)$ acting on $11$ symbols;

	\item[(iii)] $G_0=M_{11}$, or $G_0=M_{23}$, where $M_{11}$ and $M_{23}$ denote the Mathieu groups acting on $11$ and $23$ symbols, respectively.
	
	\item[(iv)]  $\PSL(d,q) \subseteq G_0 \subseteq \mathrm{P}\Sigma \mathrm{L}(d,q)$, with $d>1$ an integer, and $q$ a prime power, acting on $\ell=\frac{q^d-1}{q-1}$ points of the projective space $\mathbb{P}^{d-1}(\mathbb{F}_q)$.
\end{itemize}
Also, under the action of $G_0$ on the right cosets of $H_0:=\Gal(N_K/K)$ in $G_0$, defined as  
\begin{equation*}
	(H_0s)^{g}:=H_0sg^{-1}, \quad \forall g \in G_0,
\end{equation*}
the set $T_0$, formed by the elements of $H_0$ which have only $H_0$ as fixed point, is non empty.
%Furthermore, in the cases $(1)$ to $(3)$, except $G_0=A_5$, the quotient group $\frac{H_0}{<H_0^{\prime},T_0>}$ is trivial.
% For the case $(4)$, including $G_0=A_5=PSL(2,4)$, we have
%\vspace*{0.2cm}
%\begin{equation} \label{equation, order of H0/H0prime for G between PSl and PGammal}
%	\frac{H_0}{<H_0^{\prime},T_0>} \simeq	\left\{
%	\begin{array}{ll}
	%	0 & :  d >2 \\
	%	& \\
	%	\frac{\mathbb{Z}}{	\left(\frac{q-1}{2}\right) \mathbb{Z}} & : \text{$d=2$, and $p$ is odd} \\
	%	& \\
	%	\frac{\mathbb{Z}}{	\left({q-1}\right) \mathbb{Z}} & : d=p=2.\\
%	\end{array}
%	\right.
%\end{equation}
%where $k=(P\Gamma L (d,q) : G_0)$.  
\end{proposition}

\begin{lemma} \cite[Section 7]{Zantema} \label{lemma, R(G) is empty for cases 1 to 3}
%	With the assumptions and notations 
 In the cases (i), (ii), and (iii) of Proposition \ref{proposition, possible groups G for non-cyclic}, except the case $G_0=A_5$, the quotient group $\frac{H_0}{\langle H_0^{\prime},T_0\rangle }$ is trivial, where $H_0^{\prime}$ denotes the derived subgroup of $H_0$.
\end{lemma}

\begin{lemma} \label{lemma, Ho,T0=T0}
%With the assumptions and notations 
In  the case (iv) of Proposition \ref{proposition, possible groups G for non-cyclic}, namely 
		\begin{equation} 
		PSL(2,q) \subseteq G_0 \subseteq P\Sigma L(2,q),
	\end{equation}
for $q$ a prime power, the following assertions hold:
	\begin{itemize}
		\item[(i)] We have $\langle H_0^{\prime},T_0\rangle =\langle T_0 \rangle$, where $H_0^{\prime}$ denotes the derived subgroup of $H_0$.
		
		\item[(ii)] For $d >2$, the quotient group $\frac{H_0}{\langle T_0 \rangle}$ is trivial.
		
		\item[(iii)] For $d=2$, there exists a homomorphism
		\begin{equation} \label{equation, map psi with kernel T0}
		\psi : H_0 \rightarrow \mathbb{F}_{q}
		\end{equation}
	such that $\Ker(\psi)=\langle T_0 \rangle$, and
	\begin{equation*}
			\mathrm{Im}(\psi)	= \left\{
		\begin{array}{ll}
			\left(\mathbb{F}_q^{\times}\right)^2, &   \text{if $q$ is odd}, \\
			& \\
		\mathbb{F}_q^{\times}, &   \text{if $q$ is a power of $2$}.\\
		\end{array}
		\right.
	\end{equation*}
	\end{itemize}

\end{lemma}

\begin{proof}
	See \cite[Proof of Theorem 7.1]{Zantema}.
\end{proof}

\begin{remark}
	Note that the case $G_0=A_5=\PSL(2,4)$ is included in Lemma \ref{lemma, Ho,T0=T0}.
\end{remark}

%In addition to the above results, we need also the following lemma which reveals the image of the inertia groups under the map $\Phi$ given in Theorem \ref{theorem, the composite maps}.

Now, we can generalize Zantema's results in \cite[Corollary 7.2, Theorem 8.3]{Zantema}. This proves Conjecture \ref{conj, A5-fields} in a more general setting.

\begin{theorem} \label{theorem, our results for non-cyclic fields of prime degrees}
	%Let $K/\mathbb{Q}$ be a non-cyclic extension of a prime degree $\ell$. Then the following assertions hold:
	Let $K/\mathbb{Q}$ be a non-Galois extension of prime degree $\ell \geq 5$. Denote by $N_K$ the Galois closure of $K$ (over $\mathbb{Q}$),  and let $G_0=\Gal(N_K/\mathbb{Q})$ be unsolvable. Then
	%\begin{itemize}
	%	\item[(1)] We have 
		\begin{equation*}
			\Po(K)_{nr1}=\Po(K)_{nr}=\Po(K)=\Cl(K),
		\end{equation*}
		unless 
		\begin{equation} \label{equation, G is between Psl and PGammal}
			\PSL(2,q) \subseteq G_0 \subseteq \mathrm{P}\Sigma \mathrm{L}(2,q), %\text{$q$ is a prime power}.
		\end{equation}
	for some prime power $q$. In this case,
		%for some prime power $q$.
	%	\vspace*{0.2cm}
		 %Let $K$ be a $G$-field for which $G$ is satisfying in \eqref{equation, G is between Psl and PGammal}. Then 
		 $\ell$  is a Fermat prime, $\frac{\Cl(K)}{\Po(K)_{nr1}}$ is a cyclic group whose order dividing $\ell -2$, and %in this case we have
				\begin{equation} \label{equation, Po(K)nr1 is equal to Po(K)}
				\Po(K)_{nr1}=\Po(K)_{nr}=\Po(K).
			\end{equation} 
%	\end{itemize}

\end{theorem}

\begin{proof}
By Proposition \ref{proposition, possible groups G for non-cyclic}, the set $T_0$ is non-empty. Hence, we can apply Theorem  \ref{theorem, CH}. If $G_0$ is not satisfying \eqref{equation, G is between Psl and PGammal}, then using Lemma \ref{lemma, R(G) is empty for cases 1 to 3} and Theorem \ref{theorem, CH}, we find $\Po(K)_{nr1}=\Cl(K)$. 

Now
%Using Proposition \ref{proposition, possible groups G for non-cyclic} and first two parts of Lemma \ref{lemma, Ho,T0=T0}, one proves the assertion (1).
 suppose that 
		\begin{equation*} 
		\PSL(2,q) \subseteq G_0 \subseteq \mathrm{P}\Sigma \mathrm{L}(2,q), 
	\end{equation*}
for some prime power $q$.
	Since $[K:\mathbb{Q}]=\ell$ is prime, $q-1$ is odd, see \cite[Section 7]{Zantema}. On the other hand, $q+1=\frac{q^2-1}{q-1}=\ell$ is  prime. Hence, $q$ must be a power of $2$, which implies that $\ell$ is a Fermat prime. Theorem \ref{theorem, CH} and Lemma \ref{lemma, Ho,T0=T0} show that $\frac{\Cl(K)}{\Po(K)_{nr1}}$ is a cyclic group whose order is a divisor of $q-1=\ell -2$. It remains to show that
	\begin{equation*}
	\Po(K)_{nr1}=\Po(K).
	\end{equation*}
	%for a $G_0$-field $K$ with $G_0$ satisfying in \eqref{equation, G is between Psl and PGammal}. 
	To accomplish this, we apply Theorem \ref{theorem, the composite maps}.
	Keep the same notations introduced in Section \ref{section, image of Po(K)}. Using Lemma \ref{lemma, Ho,T0=T0}, one can write the composite map \eqref{equation, composite map Lambda o theta} as
	\begin{equation*} 
	\Po(K) \xrightarrow{\alpha_K} \frac{H}{H_1} \xrightarrow{\theta}  \frac{H}{\langle H_1,T\rangle } \xrightarrow{\lambda} \frac{H_0}{\langle T_0 \rangle}.
\end{equation*}
We show that the map
\begin{equation*}
	\Phi=\lambda \circ \theta \circ \alpha_K : \Po(K)\rightarrow \frac{H_0}{\langle T_0 \rangle},
\end{equation*}
  given in Theorem \ref{theorem, the composite maps}, is the zero map, i.e., we show that
  \begin{equation} \label{equation, map Phi is zero}
  \Phi \left([\Pi_{p^f}(K)]\right) \in \langle T_0 \rangle, \quad \forall \,  \text{$p$ prime}, \, \forall f \in \mathbb{N}.
  \end{equation}
Then the equality $\Po(K)_{nr1}=\Po(K)$ follows from part (ii) of Theorem \ref{theorem, the composite maps}.
  % Using Lemma \ref{lemma, Ho,T0=T0}, for a prime  $p$, and for every $i=1,\dots,t$,  the composite maps \eqref{equation, composite maps} can be written as
 % \begin{equation*} 
  %	\frac{	s_i Z(\mathfrak{P}/p)s_i^{-1}}{	s_i T(\mathfrak{P}/p)s_i^{-1}} \xrightarrow{\phi_i} \frac{	\left(s_i Z(\mathfrak{P}/p)s_i^{-1} H \right)}{\left(s_i T(\mathfrak{P}/p)s_i^{-1} H \right)} \xrightarrow{\psi_i} \frac{H}{H_1} \xrightarrow{\theta}  \frac{H}{\langle H_1,T>} \xrightarrow{\lambda} \frac{H_0}{\langle T_0 \rangle}.
  %\end{equation*}
Let $p$ be a prime number and fix a prime $\mathfrak{P}$ of $N_H$ above $p$, where $N_H$ denotes the Galois closure of $H_K$ (over $\mathbb{Q}$). Using the notations of Section \ref{section, image of Po(K)}, since $H_K/K$ is unramified, we have
\begin{equation} \label{equation, T cap H subset H1}
s_i T(\mathfrak{P}/p)s_i^{-1} \cap H \subseteq H_1, \quad \forall i=1,\dots,t,
\end{equation}
	see \cite[Section 8]{Zantema}. As we have seen in the proof of Theorem \ref{theorem, the composite maps}, the map $\lambda$ is induced by the canonical surjection $\pi : G \rightarrow G/H_2$. Now, similar to \cite[Lemma 3.10]{Chabert II}, one can show that
	\begin{equation} \label{equation, canonical surjection pi}
		\pi\left(T(\mathfrak{P}/p)\right)=\frac{T(\mathfrak{P}/p)}{T(\mathfrak{P}/p) \cap H_2}.
	\end{equation}
	 On the other hand
	%Denote by $T(\mathfrak{P}/\mathfrak{P}_0)$ the inertia group of $N_H/N_K$ at $\mathfrak{P}$. Then, one can show that
%	\begin{align} \label{equation, relation between inertia groups}
	%	T(\mathfrak{P}/\mathfrak{P}_0)=T(\mathfrak{P}/p) \cap \Gal(N_H/N_K)=T(\mathfrak{P}/p) \cap H_2,
%	\end{align}
	%where $T(\mathfrak{P}/\mathfrak{P}_0)$ denotes the inertia group of $N_H/N_K$ at $\mathfrak{P}$. On the other hand, 
%	and %there exists an exact sequence as follows
	%\begin{equation} \label{equation, exact sequence for inertia groups}
	%	0 \rightarrow T(\mathfrak{P}/\mathfrak{P}_0) \rightarrow T(\mathfrak{P}/p) \rightarrow T(\mathfrak{P}_0/p) \rightarrow 0,
	%\end{equation}
	\begin{equation} \label{equation, T(B) mod H2 T(B0)}
		\frac{T(\mathfrak{P}/p)}{T(\mathfrak{P}/p) \cap H_2} \simeq T(\mathfrak{P}_0/p),
	\end{equation}
see \cite[$\S 7$, Proposition 22]{Serre}.
 %\begin{equation*}
 %	\pi:G \rightarrow \frac{G}{H_2}=G_0.
% \end{equation*}
 %can be thought of as the ``taking modulo $H_2$'' function. 
 Hence by \eqref{equation, T cap H subset H1}, \eqref{equation, canonical surjection pi} and \eqref{equation, T(B) mod H2 T(B0)} we get	%Therefore, using Corollary \ref{corollary, relation between inertia groups}, we find
	%\begin{equation} \label{equation, lambda o theta of siTsi in T0}
%\lambda \circ \theta \left(s_i T(\mathfrak{P}/p)s_i^{-1} \cap H\right) \subseteq \langle T_0 \rangle.
%	\end{equation}
%	Since, as we have seen in the proof of Theorem \ref{theorem, the composite maps}, considering the class modulo $H_2=\Gal(N_H/N_K)$ is equivalent to considering the restriction to $N_K$, we get
	\begin{align} \label{equation, cap H0 sub T0}
\hat{s}_i T(\mathfrak{P}_0/p)\hat{s}_i^{-1} \cap H_0 =\lambda \circ \theta \left(s_i T(\mathfrak{P}/p)s_i^{-1} \cap H\right)	 \subseteq \langle T_0 \rangle, \quad  \forall i=1,\dots,t, %\quad \forall i=1,\dots,t,
	\end{align}
where $\hat{s}_i=s_i |_{N_K}$,  and $\mathfrak{P}_0=\mathfrak{P} \cap N_K$. By a similar argument, it can be shown that
\begin{equation*}
\lambda \circ \theta \left(s_i Z(\mathfrak{P}/p)s_i^{-1} \cap H\right)	=\hat{s}_i Z(\mathfrak{P}_0/p)\hat{s}_i^{-1} \cap H_0, \quad \forall i=1,\dots,t.
\end{equation*}
 %(Recall that $H_0=H/H_2$. Also, note that for every $\sigma \in T(\mathfrak{P}/p)$, we have $\sigma\mid_{N_K} \in T(\mathfrak{P}_0/p)$, and conversely, for every $\tau \in T(\mathfrak{P}_0/p)$, one has $\tau=\sigma\mid_{N_K}$).
Now if for every $i=1,\dots,t$, 
\begin{equation*}
%\lambda \circ \theta \left(s_i Z(\mathfrak{P}/p)s_i^{-1} \cap H\right)	=
\hat{s}_i Z(\mathfrak{P}_0/p)\hat{s}_i^{-1} \cap H_0 \subseteq \langle T_0 \rangle,
\end{equation*}
then, by Theorem \ref{theorem, the composite maps} for every positive integer $f$ we have
\begin{align*}	
%\Phi\left(s_i(\mathfrak{P}) \cap K\right)=
\Phi \left([\Pi_{p^f}(K)]\right)=\prod_{\{i | f_i=f\}} \hat{s}_i g_{\mathfrak{P}_0}^{f_i} \hat{a}_i \hat{s}_i^{-1}  \in \hat{s}_i Z(\mathfrak{P}_0/p)\hat{s}_i^{-1} \cap H_0 \subseteq \langle T_0 \rangle,  %\forall i=1,\dots,t,
\end{align*}
which implies that \eqref{equation, map Phi is zero} holds.
%$\Phi \left([\Pi_{p^f}(K)]\right) \in \langle T_0 \rangle$, for every integer $f >0$. 
So, assume that for at least one conjugate of $Z(\mathfrak{P}/p)$, say $Z(\mathfrak{P}/p)$ itself, we have
\begin{equation}
\lambda \circ \theta \left( Z(\mathfrak{P}/p) \cap H \right)	=Z(\mathfrak{P}_0/p) \cap H_0 \not \subseteq \langle T_0 \rangle.
\end{equation}
For each subgroup $B$ of $G_0$, let $\widetilde{B}:=B \cap PSL(2,q)$. Since $G_0=\frac{G}{H_2}$, by \eqref{equation, T cap H subset H1} we have
\begin{equation} \label{equation,  Ttilda is subseteq T0}
	x \widetilde{T(\mathfrak{P}_0/p)} x^{-1} \cap \widetilde{H}_0 \subseteq \widetilde{\langle T_0 \rangle}, \quad \forall x \in PSL(2,q),
\end{equation}
while
\begin{equation} \label{equation, Ztilda is not subseteq T0}
\widetilde{Z(\mathfrak{P}_0/p)} \cap \widetilde{H}_0 \not\subseteq \widetilde{\langle T_0 \rangle},
\end{equation}
because the index $\left[G_0 : \PSL(2,q)\right]$ is a $2$-power, and by part (iii) of Lemma \ref{lemma, Ho,T0=T0}, the index $\left[ Z(\mathfrak{P}_0/p) \cap H_0 : Z(\mathfrak{P}_0/p) \cap \langle T_0 \rangle \right]$ is a divisor of $\mathbb{F}_q^{\times}$. On one hand,  the groups $T(\mathfrak{P}_0/p)$ and $Z(\mathfrak{P}_0/p)$ are solvable. On the other hand, a solvable subgroup of $PSL(2,q)$ is either of the form
\begin{equation*}
	\begin{pmatrix}
	\langle a\rangle  & V \\
	0 & 1	
	\end{pmatrix},
\end{equation*}
or 
\begin{equation*}
\left\langle
	\begin{pmatrix}
	0 & 1\\
	1 & 0	
\end{pmatrix},
	\begin{pmatrix}
	a & 0\\
	0 & 1	
\end{pmatrix}
\right\rangle
\end{equation*}
for some $a \in \mathbb{F}_q^{\times}$ and some vector space $V$ over $\mathbb{F}_2(a)$ \cite[Proof of Theorem 8.3, page 193]{Zantema}. Let $\sigma \in \Gal(\mathbb{F}_q/\mathbb{F}_2)$ such that
\begin{equation*}
	G_0=\left\langle \PSL(2,q), \sigma \right\rangle .
\end{equation*}
Then one may choose coordinates (for $\mathbb{P}^{d-1}(\mathbb{F}_q)$) such that
\begin{equation*}
	H_0=\left\langle \begin{pmatrix}
		\mathbb{F}_q^{\times} & \mathbb{F}_q\\
		0 & 1	
	\end{pmatrix}, \sigma \right\rangle.
\end{equation*}
By Lemma \ref{lemma, Ho,T0=T0} (part (iii)), one can write
\begin{equation} \label{equation, matrix form for T0}
	\langle T_0 \rangle=\left\langle\begin{pmatrix}
		A & \mathbb{F}_q\\
		0 & 1	
	\end{pmatrix}, \sigma \right\rangle,
\end{equation}
for some subgroup $A$ of $\mathbb{F}_q^{\times}$. % \cite[Proof of Theorem 8.3, page 192]{Zantema}. 
 Now if
\begin{equation*}
	\widetilde{Z(\mathfrak{P}_0/p)}=s \left\langle
	\begin{pmatrix}
		0 & 1\\
		1 & 0	
	\end{pmatrix},
	\begin{pmatrix}
		a & 0\\
		0 & 1	
	\end{pmatrix}
	\right\rangle s^{-1}
\end{equation*}
for some $s \in PSL(2,q)$, then
\begin{equation*}
\widetilde{T(\mathfrak{P}_0/p)}=s
\begin{pmatrix}
	\langle a\rangle  & 0\\
	0 & 1	
\end{pmatrix}
s^{-1},
\end{equation*}
since $\widetilde{T(\mathfrak{P}_0/p)}$ is a normal subgroup of $\widetilde{Z(\mathfrak{P}_0/p)}$ with a cyclic factor group. Let $x \in \widetilde{H}_0$ and denote by $\gamma_x$ the quotient of the two eigenvalues of $x$, defined up to taking the inverse. Then
\begin{equation*}
\gamma_x	= \left\{
	\begin{array}{ll}
		1,  &  \text{if}\, \,  \, x=s \begin{pmatrix}
			0 & a^i \\
			1 & 0
		\end{pmatrix}
	s^{-1}, \\
		& \\
		a^i, & \text{if}\, \,  \, x=s \begin{pmatrix}
			a^i & 0 \\
			0 & 1
		\end{pmatrix}
	s^{-1}.
			\\
	\end{array}
	\right.
\end{equation*}
By \eqref{equation, matrix form for T0}, 
 $x \in \widetilde{\langle T_0 \rangle}$ if and only if $\gamma_x \in A$. From \eqref{equation, Ztilda is not subseteq T0}, it follows that
 %\begin{equation*} 
 %	\widetilde{Z(\mathfrak{P}_0/p)} \cap \widetilde{H}_0 \not\subseteq \widetilde{\langle T_0 \rangle},
 %\end{equation*}
 %which implies that 
 $a \not \in A$. Therefore, we obtain
 \begin{equation*}
 	s \widetilde{T(\mathfrak{P}_0/p)} s^{-1} \cap \widetilde{H}_0 \not \subseteq \widetilde{\langle T_0 \rangle},
 \end{equation*}
which contradicts the relation \eqref{equation,  Ttilda is subseteq T0}. Now assume that
\begin{equation*}
\widetilde{Z(\mathfrak{P}_0/p)}=s \begin{pmatrix}
	\langle a\rangle  & V \\
	0 & 1	
\end{pmatrix} s^{-1},
\end{equation*}
and
\begin{equation*}
	\widetilde{T(\mathfrak{P}_0/p)}=s \begin{pmatrix}
		\langle b\rangle  & W \\
		0 & 1	
	\end{pmatrix} s^{-1}
\end{equation*}
for some $s \in PSL(2,q)$, $a,b \in \mathbb{F}_q^{\times}$, and some vector spaces $V$, $W$ over $\mathbb{F}_2(a)$. Since
\begin{equation*}
	\widetilde{Z(\mathfrak{P}_0/p)} \cap \widetilde{H}_0 \not\subseteq \widetilde{\langle T_0 \rangle},
\end{equation*}
there exists an element $x \in \widetilde{Z(\mathfrak{P}_0/p)}$ for which $\gamma_x \not \in A$, while by \eqref{equation, cap H0 sub T0} we have $b \in A$. Denote by $r_a$ and $r_b$ the order of $a$ and $b$, respectively, and let $v=|V|$. Recall that, by Proposition \ref{proposition, Zantema result for phii's and psii's}, the prime $p$ will be decomposed in $K/\mathbb{Q}$ as
	\begin{equation} \label{equation, decomposition form of p}
	p \mathcal{O}_K=\prod_{i=1}^{t} (s_i(\mathfrak{P})\cap K)^{e_i},
\end{equation}
where the residue class degree of $s_i(\mathfrak{P})\cap K$ is $f_i$. Zantema \cite[Proof of Theorem 8.3]{Zantema} showed that two cases are possible:
\begin{itemize}
	\item[Case (1):] $e_1=f_1=f_2=1$, $e_2=v$, and for every $i \geq 3$, $e_i$ (resp. $f_i$) is a $2$-power times $r_b$ (resp. is a $2$-power times $r_a/r_b$);
\vspace*{0.2cm}
	\item[Case (2):]  $e_1f_1=2$, and for every $i \geq 2$, $e_i$ (resp. $f_i$) is a $2$-power times $r_b$ (resp. is a $2$-power times  $r_a/r_b$).
\end{itemize}
Let $\left(Z(\mathfrak{P}_0/p):T(\mathfrak{P}_0/p)\right)=(w.r_a)/r_b$. In Case (1), set
\begin{equation*}
\quad w_i=(w.r_a)/(r_b.f_i), \, \, \forall i \geq 3.
\end{equation*}
Likewise, in Case (2), let
\begin{equation*}
	\quad w_i=(w.r_a)/(r_b.f_i), \, \, \forall i \geq 2
\end{equation*}
(Note that $w$ and $w_i$'s are $2$-powers). By Proposition \ref{proposition, Zantema result for phii's and psii's}, and Theorem \ref{theorem, the composite maps}, we have
\begin{equation*}
	\lambda \circ \theta \left(	s_ig_{\mathfrak{P}}^{f_i}a_is_i^{-1}\right) =\hat{s}_i g_{\mathfrak{P}_0}^{f_i} \hat{a}_i \hat{s}_i^{-1} \in H_0, \quad \forall i=1,\dots,t.
\end{equation*}
Also, one can show that in Case (1) (resp. Case (2)), for every $i \geq 3$ (resp. $i \geq 2$), one has
\begin{equation*}
	\left(g_{\mathfrak{P}_0}^{f_i} \hat{a}_i\right)^{w_i} \in T(\mathfrak{P}_0/p),
\end{equation*}
since $\hat{a}_i, g_{\mathfrak{P}_0}^{f_i w_i} \in T(\mathfrak{P}_0/p)$ and $Z(\mathfrak{P}_0/p)/T(\mathfrak{P}_0/p)$ is abelian. Therefore, in Case (1) (resp. in Case (2)) for every $i \geq 3$ (resp. for every $i \geq 2$) we have
\begin{equation*} 
\left(\Phi \left((s_i(\mathfrak{P}) \cap K)\right)\right)^{w_i}=\left(\hat{s}_i g_{\mathfrak{P}_0}^{f_i} \hat{a}_i \hat{s}_i^{-1}\right)^{w_i} \in \left(\hat{s}_i T(\mathfrak{P}_0/p) \hat{s}_i^{-1}\right) \cap H_0 \subseteq \langle T_0 \rangle.
\end{equation*}
(The latter inclusion follows from the relation \eqref{equation, cap H0 sub T0}). On the one hand, $w_i$ is a $2$-power, and, on the other hand, by part (iii) of Lemma \ref{lemma, Ho,T0=T0} the order of $H_0/\langle T_0 \rangle$ is a divisor of $q-1$ (which is odd). Consequently, in Case (1) (resp. in Case (2)) for every $i \geq 3$ (resp. for every $i \geq 2$), we have
\begin{equation} \label{equation, the image of Phi is contained in T0 for i geq 3}
	\Phi \left((s_i(\mathfrak{P}) \cap K)\right)=\hat{s}_i g_{\mathfrak{P}_0}^{f_i} \hat{a}_i \hat{s}_i^{-1} \in \langle T_0 \rangle.
\end{equation}
Using the decomposition form of $p$ in $K/\mathbb{Q}$, as given in \eqref{equation, decomposition form of p}, we obtain
\begin{equation} \label{equation, Phi(p) is equal to one}
1=\Phi \left( p \mathcal{O}_K\right)=\Phi \left(\prod_{i=1}^{t} (s_i(\mathfrak{P})\cap K)^{e_i}\right) = \prod_{i=1}^t (\hat{s}_i g_{\mathfrak{P}_0}^{f_i} \hat{a}_i \hat{s}_i^{-1})^{e_i} \in \langle T_0 \rangle.
\end{equation}
Hence by the relations \eqref{equation, the image of Phi is contained in T0 for i geq 3} and \eqref{equation, Phi(p) is equal to one}, for Case (2) we obtain
\begin{equation*}
 \left(\Phi(s_1(\mathfrak{P}) \cap K)\right)^{e_1} \in \langle T_0 \rangle.
\end{equation*}
Since $e_1 \leq 2$ and the order of $H_0/\langle T_0 \rangle$ is odd, we conclude that in Case (2)
\begin{equation*}
	\Phi \left((s_1(\mathfrak{P}) \cap K) \right)=\hat{s}_1 g_{\mathfrak{P}_0}^{f_1} \hat{a}_1 \hat{s}_1^{-1}\in \langle T_0 \rangle.
\end{equation*}
Thus, in Case (2) we have
\begin{equation*}
\hat{s}_i g_{\mathfrak{P}_0}^{f_i} \hat{a}_i \hat{s}_i^{-1} \in \langle T_0 \rangle, \quad \forall i=1,\dots,t, 
\end{equation*}
which implies that
\begin{equation*}
	\Phi \left([\Pi_{p^f}(K)]\right)=\prod_{\{i | f_i=f\}} \hat{s}_i g_{\mathfrak{P}_0}^{f_i} \hat{a}_i \hat{s}_i^{-1} \in \langle T_0 \rangle, \quad \forall f \in \mathbb{N}.
\end{equation*}

Note that $r_a \neq 1$, and $\langle a\rangle  \neq \langle b\rangle $, so $r_a/r_b >1$, see \cite[Proof of Theorem 8.3]{Zantema}. Hence, in Case (1), one has $f_i > 1$, for every  $i \geq 3$. Therefore, in this case, we have
\begin{equation*}
	\Pi_{p^1}(K)=(s_1(\mathfrak{P}) \cap K)(s_2(\mathfrak{P})\cap K).
\end{equation*}
Since $V$ is a vector space over $\mathbb{F}_2(a)$, $r_a$ is a divisor of $v-1$, so  $r_a/r_b$ is also dividing  $v-1$. Hence one can write $v-1= (r_a/r_b) \alpha$ for some positive integer $\alpha$. Now from 
\begin{equation*}
	\left(Z(\mathfrak{P}_0/p):T(\mathfrak{P}_0/p)\right)=(w.r_a)/r_b,
\end{equation*}
we obtain
\begin{equation*}
	\left(\hat{s}_2 g_{\mathfrak{P}_0} \hat{a}_2 \hat{s}_2^{-1}\right)^{(v-1)w}=\left( (\hat{s}_2 g_{\mathfrak{P}_0} \hat{a}_2 \hat{s}_2^{-1})^{(w.r_a)/r_b} \right)^{\alpha} \in (\hat{s}_2 T(\mathfrak{P}_0) \hat{s}_2^{-1}) \cap H_0 \subseteq \langle T_0 \rangle,
\end{equation*}
in which the latter inclusion follows from \eqref{equation, cap H0 sub T0}. Since $w$ is a $2$-power and the order of $H_0/\langle T_0 \rangle$ is a divisor of $q-1$, we get
\begin{equation*}
\left(\hat{s}_2 g_{\mathfrak{P}_0} \hat{a}_2 \hat{s}_2^{-1}\right)^{(v-1)} \in \langle T_0 \rangle.
\end{equation*} 
Finally, using the decomposition form of $p$, as in \eqref{equation, decomposition form of p}, and the relation \eqref{equation, Phi(p) is equal to one} we have
\begin{equation*}
	(\hat{s}_1 g_{\mathfrak{P}_0} \hat{a}_1 \hat{s}_1^{-1}) (\hat{s}_2 g_{\mathfrak{P}_0} \hat{a}_2 \hat{s}_2^{-1}) (\hat{s}_2 g_{\mathfrak{P}_0} \hat{a}_2 \hat{s}_2^{-1})^{(v-1)}\left(\prod_{i=3}^t (\hat{s}_i g_{\mathfrak{P}_0}^{f_i} \hat{a}_i \hat{s}_i^{-1})^{e_i} \right) \in \langle T_0 \rangle,
\end{equation*}
which implies that
\begin{equation*}
	\Phi \left( [\Pi_{p^1}(K)]\right)= (\hat{s}_1 g_{\mathfrak{P}_0} \hat{a}_1 \hat{s}_1^{-1}) (\hat{s}_2 g_{\mathfrak{P}_0} \hat{a}_2 \hat{s}_2^{-1}) \in \langle T_0 \rangle.
\end{equation*}
The proof is complete.
%%%%%%%%%%%%%%%%%%%%%%%%%%%%%%%%%%%%%%%%%%%%%
%%%%%%%%%%%%%%%%%%%%%%%%%%%%%%%%%%%%%%%%%%%%%
%%%%%%%%%%%%%%%%%%%%%%%%%%%%%%%%%%%%%%%%%%%%%
%Then by an argument similar to Zantema's reasoning as in \cite[pages 193--195]{Zantema}, one can show that $\Phi \left([\Pi_{p^f}(K)]\right) \in \langle T_0 \rangle$, for every integer $f >0$ (It is just enough to replace ``$H_1$'' with ``$\langle T_0 \rangle$'').
%%%%%%%%%%%%%%%%%%%%%%%%%%%%%%%%%%%%%%%%%%%%%%
%%%%%%%%%%%%%%%%%%%%%%%%%%%%%%%%%%%%%%%%%%%%%%
%%%%%%%%%%%%%%%%%%%%%%%%%%%%%%%%%%%%%%%%%%%%%%
\end{proof}

\begin{remark}
	Applying the above theorem for $G_0=A_5 \simeq \PSL(2,4)$, we get
	\begin{equation*}
\# \frac{\Cl(K)}{\Po(K)}=\#\frac{\Cl(K)}{\Po(K)_{nr}}=1 \, \text{or} \,\, 3,
	\end{equation*}
which has already obtained by Chabert and Halberstadt \cite[Proposition 6.8, Part (ii)]{Chabert II}.  
%The above theorem shows that for a $A_5$-field $K$, the order of the factor group $\frac{\Cl(K)}{\Po(K)_{nr1}}$ divides $3$, since $A_5 \simeq PSL(2,4)$. Thus, we obtain Chabert and Halberstadt's result in \cite[Proposition 6.8, Part (ii)]{Chabert II}.  
\end{remark}

\section{Proof of Chabert and Halberstadt's conjecture Concerning $D_4$-fields} \label{Section, even dihedral fields}
In this section, we show that not only  Conjecture \ref{conj, D4-fields} is true, but more generally, such assertion holds for all $D_n$-fields $K$, where $n \geq 4$ is an even integer. % and $D_n$ denotes the dihedral group of order $2n$, see Theorem \ref{theorem, our results for Dn-fields}. 
 The proof is only based on the decomposition forms of primes in $K/\mathbb{Q}$, which is an entirely different approach from the ones used in the previous sections.

\begin{theorem} \cite[Theorem 5.2]{Chabert II}, \cite[Theorem 2.2]{M4} \label{theorem, CH-Dn fields}
	Let $K$ be a $D_n$-field, where $n \geq 4$ is an even integer. Denote  by $E$ the unique subfield of $K$ of degree $\frac{n}{2}$ over $\mathbb{Q}$. Then 
	\begin{equation*}
		\Po(K)_{nr}=\epsilon_{K/E}(\Cl(E)),
	\end{equation*}
	where
	$\epsilon_{K/E}:\Cl(E) \rightarrow \Cl(K)$
	denotes the capitulation map.
	%following assertions hold:
	%\begin{itemize}
	%	\item[(1)] \cite[Theorem 5.2]{Chabert II}, \cite[Theorem 2.2]{M4} 
	
	%%	\item[(2)] \cite[Corollary 2.4]{M4} If $h_K$ is odd, then
	%	\begin{equation*}
		%		\Po(K)_{nr}=\Po(K) \simeq \Cl(E).
		%	\end{equation*} 
	%\end{itemize}
\end{theorem}

%In the special case, for $n=4$, Chabert and Halberstadt obtained the following result.

%\begin{theorem} \cite[Proposition 6.4]{Chabert II}
%	Let $K=\mathbb{Q}(\theta)$ be a $D_4$-field, where $\theta$ is a root of the polynomial
%	\begin{equation*}
	%		f(X)=X^4+aX^2+b,
	%	\end{equation*}
%	with $a,b \in \mathbb{Q}$,  $\sqrt{b},\sqrt{d},\sqrt{bd} \not \in \mathbb{Q}$ ($d=a^2-4b$). If $H_K=K(\sqrt{d})$, then $\Po(K)_{nr}=0$. 
%\end{theorem}

%The above result has led to the conjecture below.

%We show that the above conjecture is true in a more general case.
%holds in a more general case.

\begin{theorem} \label{theorem, our results for Dn-fields}
	Let $K$ be a $D_n$-field, where $n \geq 4$ is an even integer. If $h_K=2$, then $\Po(K)_{nr}=0$. 
\end{theorem}

%\begin{theorem}
%	Let $K$ be a $D_4$-field. If $h_K=2$, then $\Po(K)_{nr}=0$.
%\end{theorem}

\begin{proof}
	Denote by $E$ the unique  subfield of $K$ of degree $\frac{n}{2}$ over $\mathbb{Q}$. If the capitulation map is zero, i.e.,  $\epsilon_{K/E}(\Cl(E))=0$, then the assertion follows from Theorem \ref{theorem, CH-Dn fields}. So, let $\epsilon_{K/E}(\Cl(E))\neq 0$, which implies that $h_E \neq 1$. On the other hand, the norm map
	\begin{equation*}
		N_{K/E}: \Cl(K) \rightarrow \Cl(E)
	\end{equation*}
	is surjective, and consequently $h_E \mid h_K$, see \cite[Lemma 2.1]{M4}. Since $h_K$ is equal to $2$, so is $h_E$. Therefore, the norm map $N_{K/E}$ % and the capitulation map
	%\begin{equation*}
	%	 \epsilon_{K/E} : \Cl(E) \rightarrow \Cl(K)
	%\end{equation*}
	would be injective. Let $p$ be a prime number unramified in $K/\mathbb{Q}$. Then the decomposition form of $p$ in $K/\mathbb{Q}$ can be written as
	\begin{equation} \label{equation, decomposition form of p in Dn}
		p \mathcal{O}_K =\mathfrak{P}_1 \mathfrak{P}_2 \mathfrak{P}_3 \dots \mathfrak{P}_{\frac{n}{2}+1},
	\end{equation}
	where $f(\mathfrak{P}_1/p)=f(\mathfrak{P}_2/p)=1$, and $f(\mathfrak{P}_i/p)=2$ for  $i=3,\dots, \frac{n}{2}+1$, see \cite[Proof of Proposition 5.1]{Chabert II}. Therefore 
	\begin{equation*}
		\Pi_{p}(K)=\mathfrak{P}_1 \mathfrak{P}_2.
	\end{equation*}
	We show that the ideal $\Pi_{p}(K)$ is principal. Let $\mathfrak{p}=\mathfrak{P}_1 \cap E$. Since $p=\mathfrak{p} \cap \mathbb{Z}$ is unramified in $K$, so is $\mathfrak{p}$. Also, from
	\begin{equation*}
		1=f(\mathfrak{P}_1/p)=f(\mathfrak{P}_1/\mathfrak{p}) f(\mathfrak{p}/p)
	\end{equation*}
	we obtain $f(\mathfrak{P}_1/\mathfrak{p})=f(\mathfrak{p}/p)=1$. %Therefore, $\mathfrak{p}$ splits in $K$. 
	Since $[K:E]=2$, by the decomposition form \eqref{equation, decomposition form of p in Dn}, we must have 
	\begin{equation*}
		\mathfrak{p} \mathcal{O}_K=\mathfrak{P}_1 \mathfrak{P}_j
	\end{equation*}
	for some $j=2,3,\dots,\frac{n}{2}+1$. If $j \geq 3$, then
	\begin{equation*}
		2=f(\mathfrak{P}_j/p)=f(\mathfrak{P}_j/\mathfrak{p})f(\mathfrak{p}/p)
	\end{equation*}
	and we reach a contradiction. Therefore, 
	\begin{equation*}
		\mathfrak{p} \mathcal{O}_K=\mathfrak{P}_1 \mathfrak{P}_2=\Pi_{p}(K).
	\end{equation*}
	Hence,
	\begin{equation*}
		N_{K/E}(\Pi_{p}(K))=N_{K/E}(\mathfrak{P}_1 \mathfrak{P}_2)=\mathfrak{p}^2
	\end{equation*}
	is a principal ideal since $h_E=2$. But the norm map $N_{K/E}$ is injective, which implies that $\Pi_{p}(K)$ is principal, as claimed. Consequently $\Po(K)_{nr1}=0$. Now the triviality of $\Po(K)_{nr}$ follows from  $\Po(K)_{nr}=\Po(K)_{nr1}$  \cite[Proposition 5.1]{Chabert II}.
\end{proof}

\bibliographystyle{amsplain}

	\end{document}